\documentclass[article]{elsarticle}

\usepackage{lineno,hyperref}

\usepackage{amsmath,amsxtra,amssymb,latexsym,epsfig,amscd,amsthm,fancybox,epsfig}
\usepackage[mathscr]{eucal}
\usepackage{graphicx}
\usepackage{epsfig} % for postscript graphics files
\usepackage{epstopdf}
\usepackage{cases}
\usepackage{subfig}
\usepackage{color}

\newtheorem{thm}{Theorem}[section]
\newtheorem {asp}{Assumption}[section]
\newtheorem{lm}{Lemma}[section]

\newtheorem{deff}{Definition}[section]

\newtheorem{prop}{Proposition}[section]
\theoremstyle{definition}

\theoremstyle{remark}
\newtheorem{rem}{Remark}

\numberwithin{equation}{section}

\newcommand{\eps}{\varepsilon}

\newcommand{\A}{\mathcal{A}}

\newcommand{\C}{\mathcal{C}}

\newcommand{\M}{\mathcal{M}}
\newcommand{\F}{\mathcal{F}}

\newcommand{\E}{\mathbb{E}}

\newcommand{\N}{\mathbb{N}}

\newcommand{\PP}{\mathbb{P}}

\newcommand{\K}{\mathcal{K}}

\newcommand{\R}{\mathbb{R}}

\newcommand{\Lom}{\mathcal{L}}

\newcommand{\abs}[1]{\left\vert#1\right\vert}

\newcommand{\norm}[1]{\left\Vert#1\right\Vert}

\renewcommand{\vec}[1]{\mathbf{#1}}

\newcommand{\bdelta}{\boldsymbol{\delta}}
\newcommand{\Vphi}{\boldsymbol\varphi}
\newcommand{\Vrho}{\boldsymbol\rho}
\newcommand{\Bphi}{\boldsymbol\phi}
\numberwithin{equation}{section}

\newcommand{\bed}{\begin{displaymath}}
\newcommand{\eed}{\end{displaymath}}
\newcommand{\bea}{\bed\begin{array}{rl}}
\newcommand{\eea}{\end{array}\eed}

\newcommand{\barray}{\begin{array}{ll}}
\newcommand{\earray}{\end{array}}
\newcommand{\diag}{{\rm diag}}
\newcommand{\conv}{{\rm Conv}}

\newcommand{\1}{\boldsymbol{1}}

\def\bar{\overline}
\def\hat{\widehat}
\def\wdt{\widetilde}

\def\a.s{\text{\;a.s.\;}}

\def\bdelta{\boldsymbol{\delta}}
\def\Ephi{\E_{\Bphi}}
\def\PPphi{\PP_{\Bphi}}
\newcommand{\ko}{Kolmogorov }

\newcommand{\DD}{\mathbb{D}}
\newcommand{\BF}{\mathbb{F}}

\begin{document}
%\maketitle
\begin{frontmatter}
	
	\title{Stochastic Functional Kolmogorov Equations (I): Persistence\tnoteref{mytitlenote}}
	\tnotetext[mytitlenote]{The research of D. H. Nguyen was supported in part by the National Science Foundation under grant under grant DMS-1853467. 
		The research of N. N. Nguyen and G. Yin was supported in part by the National Science Foundation under grant under grant DMS-2114649.}
	\author[myaddress1]{Dang H. Nguyen}
	\ead{dangnh.maths@gmail.com.}
	\author[myaddress]{Nhu N. Nguyen}
	\ead{nguyen.nhu@uconn.edu}
	\author[myaddress]{George Yin\corref{mycorrespondingauthor}}
	\cortext[mycorrespondingauthor]{Corresponding author}
	\ead{gyin@uconn.edu}
	
	\address[myaddress1]{Department of Mathematics, University of Alabama, Tuscaloosa, AL
		35487, USA}
	\address[myaddress]{Department of Mathematics, University of Connecticut, Storrs, CT
		06269, USA}

\begin{abstract}
This work (Part (I)) together with its companion (Part (II) \cite{NNY21-2})
develops a new framework
for stochastic functional Kolmogorov equations, which are nonlinear stochastic differential equations depending on the current as well as the past states.
Because of the complexity of the results, it seems to be instructive to divide our contributions to two parts.
In contrast to the existing literature,
our effort is to  advance the knowledge  by allowing
delay and past dependence, yielding essential utility to a wide range of applications.
A long-standing
question of fundamental importance
pertaining to biology and ecology
is: What are the minimal necessary and sufficient conditions for long-term persistence and extinction (or for long-term coexistence of interacting species)
of a population?
Regardless of the particular applications encountered, persistence and extinction are properties shared by  Kolmogorov systems.
While there are many excellent treaties of stochastic-differential-equation-based Kolmogorov equations,
the work on stochastic Kolmogorov equations with past dependence is still scarce. Our aim here is to answer the aforementioned basic question.
This work, Part (I), is devoted to characterization of persistence, whereas its companion, Part (II) \cite{NNY21-2}, is devoted to extinction.
The main techniques used in this paper include the newly developed functional It\^o formula and asymptotic coupling and Harris-like theory for infinite dimensional systems specialized to functional equations.
General theorems for stochastic functional Kolmogorov equations are developed first. Then
a number of applications are examined
to obtain new results substantially
covering, improving, and extending
the existing literature.
Furthermore, these conditions reduce to that of
Kolmogorov systems
when there is no past dependence.



 \bigskip
 \noindent {\bf Running Title.} Stochastic Functional Kolmogorov Equations

\end{abstract}
\end{frontmatter}
%\newpage
\tableofcontents

\section{Introduction}
This work develops a novel framework
of systems of stochastic functional \ko equations.
Our main motivation stems from a wide variety of applications in ecology and biology. A long-standing
question of fundamental importance pertaining to biology and ecology
 is: What are the minimal (necessary and sufficient) conditions for long-term persistence and extinction  (or for long-term coexistence of interacting species) of a population?
 It turns out that  persistence and extinction are phenomena
go far beyond biological and ecological systems. In fact,
  such long-term properties are shared by all processes of \ko type.
 We focus on the issues for such systems that involve stochastic disturbances and past dependence in the dynamics. The problems are substantially more difficult compared to systems without delay or past independence because one has to treat infinite dimensional processes.

 An $n$-dimensional deterministic \ko system
 is an autonomous system of equations
 to depict the dynamics of $n$ interacting populations, which
takes the form
 \begin{equation}\label{ko-det}
 \dot x_i(t)=x_i(t) f_i(x_1(t), \dots, x_n(t)), \ i=1,\dots, n,
\end{equation}
where $f_i(\cdot)$ are  functions satisfying suitable conditions. Realizing that fluctuations of the environment make the dynamics of populations inherently stochastic, much effort has been placed on the study of stochastic \ko equations.
As an example, consider a simple
2-dimensional  \ko equations with stochastic effects:
\begin{equation}\label{kol-2d-ran}
\begin{cases}
dx(t)=x(t) f_1(x(t), y(t))dt+x(t)g_1(x(t), y(t))dB_1(t),\\
dy(t)=y(t) f_2(x(t), y(t))dt+y(t)g_2(x(t), y(t))dB_2(t),
\end{cases}
\end{equation}
where $B_1(t)$ and $B_2(t)$ are two Brownian motions
(independent or not). The formulation readily generalizes to $n$-dimensional stochastic \ko equations, which
 are used extensively in the modeling and analysis of ecological and biological systems such as Lotka-Volterra
predator-prey models, Lotka-Volterra competitive models, replicator dynamic systems, stochastic epidemic models, and stochastic chemostat models, among others.
The study of such systems has encompassed the central issues of persistence and extinction as well as the existence of invariant measures.
Apart from ecological and biological systems,
 numerous problems arising in mathematical physics, statistical mechanics, and many related fields, use \ko nonlinear stochastic differential equations. We mention a simple
 one-dimensional generalized Ginzburg-Landau equation
\begin{equation}\label{gl}
d x(t)=x(t)[a(t)-b(t)x^k(t)]
+ x(t)\sigma(x(t)) dB(t), \ x(0)=x_0>0,\end{equation}
where $k \ge 2$ is a positive integer, $B(t)$ is a real-valued Brownian motion.
Such equations have been used in the theory of bistable systems, chemical turbulence, phase transitions in non-equilibrium systems, nonlinear, optics with dissipation, thermodynamics, and hydrodynamic systems, etc.

Because of its prevalence in applications,
\ko systems have  attracted much attention in the past decades;
substantial progress has been made.
To proceed, let us briefly recall some of the developments to date.
Some of the
early mathematical formulations
were introduced by
Verhulst \cite{Ver38} for logistic models,
by Lotka and Volterra \cite{Lot25,Vol26} for Lotka-Volterra systems, and
by Kermack and McKendrick \cite{Ker27,Ker32} for infectious diseases modeling
using ordinary differential
equations in the last century. The study on  mathematical models
has stimulated subsequent work with attention devoted to
analyzing and predicting the
behavior of the populations in a longtime horizon.
Subsequently, not only deterministic systems, but also stochastic systems have been studied.
Resurgent effort has been devoted to finding the corresponding classification by means of threshold levels.
Fast forward,
Imhof studied long-run behavior of the stochastic replicator dynamics in \cite{Imh05}, whereas
 Hofbauer and  Imhof concentrated on time averages, recurrence, and transience for stochastic replicator dynamics in \cite{Imh09}.
 By now, \ko
  stochastic population systems (using stochastic differential equations or difference equations)
 together with their longtime behavior have been relatively well understood; see
\cite{BL16, RS14, SBA11} for  \ko stochastic systems in compact domains and
\cite{Ben14, HN18} for  certain general \ko systems in non-compact domains. Variants of \ko systems such as epidemic models \cite{Die16,DDN19, DN18,DYZ20}, tumor-immune systems \cite{TNY20} and chemostat models \cite{NNY19} etc. have also been studied. In contrast to  numerous papers that used Lyapunov function methods to analyze the underlying systems with limited success,
Bena\"im \cite{Ben14}, Bena\"im and Lobry \cite{BL16},  Bena\"im and Strickler \cite{BS19},
Chesson and Ellner \cite{CE89},
 Evans, Henning, and Schreiber \cite{EHS15}, and
Schreiber and Bena\" im \cite{SBA11} initiated the study
by examining the
corresponding
boundary behavior and considered the stochastic rate of growth;
see also
Du, Nguyen, and Yin \cite{DuNY16}.
For the most recent
 development and substantial progress, we refer to
Bena\"im \cite{Ben14},
Henning and Nguyen
\cite{HN18}, Schreiber and Bena\" im \cite{SBA11}, and references therein.

Our study in this work is to consider a class of $n$-dimensional stochastic functional \ko systems;
our effort is
to substantially advance the existing literature by allowing
 delay and past dependence, which in turn, provides essential utility to a wide range of applications.
Why is it important to consider systems with delays as well as stochastic functional \ko systems?
Mainly, the delays or past dependence are unavoidable in natural phenomena and dynamical systems;  the framework of
 stochastic functional differential equations
 is more realistic, more effective, and more general
for the population dynamics in real life
than a stochastic differential equation counterpart.
In population dynamics, some  delay mechanisms studied in the literature include age structure, feeding times, replenishment or regeneration time for resources  \cite{JC13}.
Although there are many excellent treatises of \ko stochastic differential equations, the work on \ko stochastic differential equations with delay is relatively scarce. A few exceptions are
the study on stochastic delay Lotka-Volterra competitive models
\cite{Mao04,Liu17}, the work on
stochastic delay Lotka-Volterra predator-prey models \cite{Gen17,Liu13,Liu17-dcds,Liu17-1,Wu19}, the treatment of
stochastic delay epidemic SIR models \cite{Che09,QLiu16-1,QLiu16,QLiu16-2,Mah17},
and the study on
stochastic delay chemostat models \cite{Sun18,Sun18-1,Zha19}.
Nevertheless, other than the specific models and applications treated, there has not been a unified framework
 and a systematic treatment for \ko stochastic functional differential systems yet.
Moreover,
most of the existing results involving delay are not as sharp as desired.
Our effort in this paper takes up the aforementioned issues.

It should be noted that from stochastic \ko differential equation-type models  to that of stochastic functional differential equation models requires a big leap. There are certain essential difficulties. While the solutions of stochastic differential equations are Markovian processes, the solutions of stochastic differential equations with delay is non-Markov. Then one uses the so-called segment processes for the delay equations. However, such segment processes live in an infinite dimensional space. Many of the known results in the usual stochastic differential equation setup  are no longer applicable.  Because \ko systems are highly nonlinear, analyzing such systems with delay becomes even more difficult.  New methods and techniques need to be developed to carry out the  analysis.  This brings us to the current work.

In this paper,
we first set up the problem  in a unified form,
 introduce the methodology to characterize the longtime behavior of the underlying system, to
  establish general results for the persistence, and to demonstrate the utility in a number of applications arising in ecology and biology. Our goal is to obtain
  sharp results under mild and verifiable conditions, which
 is useful for a wide variety of stochastic functional \ko systems.
 In view of the progress
  and challenges, this work combines the techniques of  functional analysis (in particular, the functional It\^o formula) in \cite{CF10,CF13}, stochastic differential equations (SDEs) in infinite dimension, as well as
 the methods of asymptotic couplings
 \cite{Ha09}, to develop a new framework for treating functional \ko systems.
 It will
 substantially generalize the methods in \cite{Ben14, HN18}.
 Our results will  cover, improve, and outperform
 the aforementioned existing results for \ko systems with and without delays. It should be mentioned that
 in the case of replicator dynamics, it seems to be no investigation of delayed stochastic systems to date to the best of our knowledge.

 Although the models with functional stochastic differential equations are more realistic and more general, the analysis of such systems become far more difficult.
Perhaps, part of the difficulties in studying stochastic delay systems is that there had been virtually no bona fide
operators and functional
It\^o formulas except some general setup in a Banach space such as \cite{SE86} before 2009.
In \cite{Dup09}, Dupire generalized the It\^o formula to a functional setting by using  pathwise functional derivatives.
The It\^o formula  developed has substantially eased the difficulties and encouraged subsequent development with a wide range of applications.
His work was developed further by Cont and Fourni\'e \cite{CF10, CF13}.
Using the newly developed functional It\^o formula
enables us to analyze effectively the segment processes in
the stochastic functional \ko equations.

  Because of the non-Markovian property of the solution processes due to delay and the use of
  memory segment functions, one needs to analyze the corresponding stochastic equations in an infinite dimensional space.
Handling  occupation measures in an infinite dimensional space to obtain the tightness and characterize its limit is more challenging, so is to prove the uniqueness of the invariant probability measure.
The associated Markov semigroups are often not strong Feller, even in some simple cases.
Because of the absence of the strong Feller property,  Doob's method to prove the uniqueness of the invariant probability measure is no longer applicable; see \cite{RRG06}. There are some recent works on asymptotic analysis for functional stochastic differential
equations; for example, see \cite{BYY16-b} and references therein.
Most notably, in \cite{Ha09},
 Hairer, Mattingly, and Scheutzow
 developed a necessary and sufficient condition for the uniqueness of the invariant probability measure using asymptotic couplings, provided sufficient conditions for weak convergence to the invariant probability measure,
and obtained
a Harris-like theory for general infinite-dimensional spaces.
By using ideas from this abstract theory and our subtle estimates for certain coupled systems, we are able to prove the uniqueness of the invariant probability measure for the \ko
systems.
To characterize
%(almost completely)
the longtime behavior of the underlying system under natural conditions and to develop a systematic method for this kind problem, we use the intuition from dynamical system theory, in which we need to examine the corresponding problem on the boundary and
reveal
%advance out our knowledge about
the behavior of the process when it is close to the boundary.
Nevertheless,
the behavior of solutions near the boundary for functional \ko systems requires more delicate analysis than
that for systems without delay.
%one might take into account situations when
Even if
the current state is close to the boundary, its history may not be.

The rest of the paper is organized as follows.
Section \ref{sec:res} presents the formulation of our problem as well as mathematical definitions and terminologies; and proceeds to state our main results of the persistence of the
stochastic functional \ko systems.
Section \ref{sec:key} examines basic properties of \ko equations with delays, including
 well-posedness of the system,
 positivity of solutions.
 %, as well as dependence on initial data of the solutions.
 Also obtained are
the tightness of families of occupation measures and the convergence to the corresponding invariant probability measures. To obtain the desired theory, a number of key
auxiliary  results are provided.
%Upon settling with key technical results,
Then the conditions for
persistence of \ko systems are given in Section \ref{sec:per}.
Finally, Section \ref{sec:app} provides
several
applications involving \ko dynamical systems and detailed account on how to use our results on stochastic functional \ko equations to treat each of the application examples.

\section{Main Results}\label{sec:res}
To help the reading, we first provide a glossary of
 symbols and notation to be used
in this paper.

\begin{tabbing}
\hspace*{0.7in}\mbox{ }\=  \kill
$r$\> a fixed positive number\\
$\abs{\cdot}$\>  Euclidean norm\\
$\C([a; b];\R^n)$\> set of $\R^n$-valued continuous
functions defined on $[a; b]$\\
$\C$\> $:= \C([-r; 0];\R^n)$\\
$\Vphi$\> $=(\varphi_1,\dots,\varphi_n)\in \C$\\
%$\varphi,\varphi_i$\> $\in \C([-r,0],\R)$\\
$\vec x$\>$=(x_1,\dots,x_n):=\Vphi(0)\in\R^n$\\
$\norm{\Vphi}$\> $:=
\sup\{\abs{\Vphi(t)}: t\in [-r,0]\}$\\
$\vec X_t$\> $:=\vec X_t(s):= \{\vec X(t + s) : -r\leq s\leq 0\}$ (segment function)\\
$X_{i,t}$\> $:=X_{i,t}(s):= \{X_i(t + s) : -r\leq s\leq 0\}$\\
$\C_+$\> $:=\left\{\Vphi=(\varphi_1,\dots,\varphi_n)\in\C: \varphi_i(s)\geq 0\;\forall s\in[-r,0],i=1,\dots,n\right\}$\\
$\partial\C_+$\> $:=\left\{\Vphi=(\varphi_1,\dots,\varphi_n)\in\C: \norm{\varphi_i}=0\text{ for some }i=1,\dots,n\right\}$\\
$\C^{\circ}_+$\> $:=\left\{\Vphi\in\C_+: \varphi_i(s)> 0,\forall s\in[-r,0], i=1,\dots,n\right\}\neq \C_+\setminus\partial\C_+$\\
$\norm{\Vphi}_\alpha$\> $:=\norm{\Vphi}+\sup_{-r\leq s< t\leq 0}\frac{\abs{\Vphi(t)-\Vphi(s)}}{(t-s)^\alpha},$ for some $0<\alpha<1$\\
$\C^\alpha$\>  space of  H\"older continuous functions
endowed with the norm $\norm{\cdot}_\alpha$\\
$\Gamma$\> $n\times n$ matrix\\
$\Gamma^\top$\> transpose of $\Gamma$\\
$\vec B(t)$\> $=(B_1(t),\dots, B_n(t))^\top$, a $n$-dimensional standard Brownian motion \\
$\vec E(t)$\> $=(E_1(t),\dots, E_n(t))^\top:=\Gamma^\top\vec B(t)$\\
$\Sigma$\> $=(\sigma_{ij})_{n\times n}:= \Gamma^\top \Gamma$\\
$\M$\>set of ergodic invariant probability measures of $\vec X_t$ supported on $\partial \C_+$\\
$\conv(\M)$\>convex hull of $\M$\\
$\vec 0$\> the zero constant function in $\C$\\
$\bdelta^*$\>the Dirac measure concentrated at $\vec 0$\\
$\1_A$\> the indicator function of set $A$\\
$D_{\eps,R}$\>$:=\left\{\Vphi\in\C_+:\norm{\Vphi}\leq R, x_i\geq \eps\;\forall i;\vec x:=\Vphi(0)\right\}, \eps, R>0$\\
$\DD$\> space of Cadlag functions mapping $[-r,0]$ to $\R^n$\\
$A_0,A_1,A_2$\>constants satisfying Assumption \ref{asp1}\\
$\gamma_0,\gamma_b,M$\>constants satisfying Assumption \ref{asp1}\\
$\vec c,h(\cdot),\mu$\>vector, function and probability measure satisfying Assumption \ref{asp1}\\
$\wdt K,b_1,b_2$\> constants satisfying Assumption \ref{asp-2}\\
$h_1(\cdot),\mu_1$\>function and probability measure satisfying Assumption \ref{asp-2}\\
$D_0,d_0$\>constants satisfying Assumption \ref{asp-unique-ipm}\\
%$B_0,B_1,B_2,B_3,p_2$\> constants satisfying Assumption \ref{a.extn2}\\
$I$\> a subset of $\{1,\dots,n\}$\\
$I^c$\>:=$\{1,\dots,n\}\setminus I$\\
$
\C_+^{I}$\>$:=\left\{\Vphi\in\C_+: \norm{\varphi_i}=0\text{ if }i\in I^c \right\}$\\
$
\C_+^{I,\circ}$\>$:=\left\{\Vphi\in\C_+: \norm{\varphi_i}=0\text{ if }i\in I^c \text{ and } \varphi_i(s)>0\;\forall s\in [-r,0]\text{ if }i\in I\right\}
$\\
$\partial \C_+^I$\>$:=\left\{\Vphi\in\C_+: \norm{\varphi_i}=0\text{ if }i\in I^c\text{ and }\norm{\varphi_i}=0\text{ for some }i\in I\right\}$\\
$\M^I$\> sets of ergodic invariant probability measures on $\C^I_+$\\
$\M^{I,\circ}$\> sets of ergodic invariant probability measures on $\C^{I,\circ}_+$\\
$\partial\M^{I}$\> sets of ergodic invariant probability measures on $\partial\C^I_+$\\
$I_\pi$\> the subset of $\{1,\dots,n\}$ such that $\pi(\C_+^{I_\pi,\circ})=1, \pi\in\M$\\
$\gamma,p_0,A$\>constants satisfying the condition in Lemma \ref{LV}\\
$\Vrho$\> $=(\rho_1,\dots,\rho_n)$ vector satisfying the condition in Lemma \ref{LV}\\
$ V_{\Vrho}(\Vphi)$\>$:=\Big(1+\vec c^\top\vec x\Big)\prod_{i=1}^nx_i^{\rho_i}\exp\left\{A_2\int_{-r}^0\mu(ds)\int_s^0 e^{\gamma(u-s)}h\big(\Vphi(u)\big)du\right\}$\\[0.5ex]
$ V_{\vec 0}(\Vphi)$\>$:=\Big(1+\vec c^\top\vec x\Big)\exp\left\{A_2\int_{-r}^0\mu(ds)\int_s^0 e^{\gamma(u-s)}h\big(\Vphi(u)\big)du\right\}$\\[0.5ex]
$\C_{V, M}$\>$:=\{\Vphi\in\C_+: A_2\gamma \int_{-r}^0\mu(ds)\int_s^0e^{\gamma(u-s)}h\big(\Vphi(u)\big)du\leq A_0, |\Vphi(0)|\leq M\}$\\
$H_1$\>constant satisfying \eqref{t-eq7} and \eqref{h2-eq0}\\
$\Vrho^*,\kappa^*$\>vector and constant satisfying \eqref{p-eq1}\\
$n^*$\> constant satisfying $\gamma_0 (n^*-1)- A_0>0$\\
$p_1$\> constant satisfying condition \eqref{cd-p0} and $p_1>p_0$\\
$H_1^*$\>constant determined in \eqref{H1*}\\
$R_{V,M}$\>constant determined in \eqref{p-l2-eq10}\\
$\eps^*$\>constant determined in \eqref{p-l2-eq8}\\
$T^*,\hat\delta$\> constants determined in Lemma \ref{p-l2}\\
$\C_V(\hat\delta)$\>$:=\{\Vphi\in\C^{\circ}_+\cap\C_{V,M} \text{ and }|\varphi_i(0)|\leq\hat\delta \text{ for some }i\}$
\end{tabbing}

Consider a stochastic delay Kolmogorov system
\begin{equation}\label{Kol-Eq}
\begin{cases}
dX_i(t)=X_i(t)f_i(\vec X_t)dt+X_i(t)g_i(\vec X_t)dE_i(t), \quad i=1,\dots,n,\\
\vec X_0=\Bphi\in \C_+,
\end{cases}
\end{equation}
and denoted by $\vec X^{\Bphi}(t)$ its solution.
 For convenience, we  usually suppress the superscript ``$\Bphi$'' and
use $\PPphi$ and $\Ephi$ to denote the probability and expectation given the initial value $\Bphi$, respectively. We also assume that the initial value is non-random. Denoted by $\{\F_t\}_{t\geq 0}$
the filtration satisfying the usual conditions and assume that the $n$-dimensional Brownian motion $\vec B(t)$ is adapted to $\{\F_t\}_{t\geq 0}$.
Note that a segment process is also referred to as a memory
segment function.
Throughout the rest of the paper, we assume the following assumptions hold.

\begin{asp}\label{asp1}
{\rm
The coefficients of \eqref{Kol-Eq} satisfy the following conditions.
\begin{itemize}
\item [\rm(1)] $\diag(g_1(\Vphi),\dots,g_n(\Vphi))\Gamma^\top\Gamma\diag(g_1(\Vphi),\dots,g_n(\Vphi))=(g_i(\Vphi)g_j(\Vphi)\sigma_{ij})_{n\times n}$ is a positive definite matrix for any $\Vphi\in\C_+$.
\item [\rm(2)] $f_i(\cdot),g_i(\cdot): \C_+\to\R$ are Lipschitz continuous in each bounded set of $\C_+$ for any $i=1,\dots,n.$
\item [\rm(3)] There exist $\vec c=(c_1,\dots,c_n)\in \R^{n}$ with $c_i>0,\forall i$, and $\gamma_b,\gamma_0>0$, $A_0>0$, $A_1>A_2>0$, $M>0$, a continuous function $h:\R^n\to\R_+$ and a probability measure $\mu$ concentrated on $[-r,0]$ such that for any $\Vphi\in\C_+$
\begin{equation}\label{asp1-3}
\begin{aligned}
\dfrac{\sum_{i=1}^n c_ix_if_i(\Vphi)}{1+\vec c^\top \vec x}&-\dfrac12\dfrac{\sum_{i,j=1}^n\sigma_{ij}c_ic_jx_ix_jg_i(\Vphi)g_j(\Vphi)}{(1+\vec c^\top \vec x)^2}+\gamma_b\sum_{i=1}^n(\abs{f_i(\Vphi)}+g_i^2(\Vphi))
\\&\leq A_0 \1_{\left\{\abs{\vec x}< M\right\}}-\gamma_0-A_1h(\vec x)+A_2\int_{-r}^0h\big(\Vphi(s)\big)\mu(ds),
\end{aligned}
\end{equation}
where $\vec x:=\Vphi(0).$ We assume without loss of generality that $h:\R^n\to[1,\infty)$,
otherwise, we can always change $\gamma_0$ and $A_1, A_2$ to fulfill this requirement.
%get so.
\end{itemize}
}
\end{asp}

\begin{asp}\label{asp-2}
{\rm
One of following assumptions holds:
\begin{itemize}
\item [\rm(a)] There is a constant $\wdt K$ such that for any $\Vphi\in \C_+$, $\vec x=\Vphi(0)$
\begin{equation}\label{A1}
\sum_{i=1}^n\abs{f_i(\Vphi)}+\sum_{i=1}^n g_i^2(\Vphi)\leq \wdt K \Big[h(\vec x)+\int_{-r}^0 h(\Vphi(s))\mu(ds)\Big].
\end{equation}
\item [\rm(b)] There exist constants $b_1,b_2>0$,
a function $h_1:\R^n\to[1,\infty]$, and
a probability measure $\mu_1$ on $[-r,0]$ such that for any $\Vphi\in \C_+$, $\vec x=\Vphi(0)$
\begin{equation}\label{A2}
b_1h_1(\vec x)\leq \sum_{i=1}^n\abs{f_i(\Vphi)}+\sum_{i=1}^n g_i^2(\Vphi)\leq b_2\Big[h_1(\vec x)+\int_{-r}^0h_1(\Vphi(s))\mu_1(ds)\Big].
\end{equation}
\end{itemize}
}
\end{asp}

\begin{rem}

Let us comment
on the  above  assumptions.
\begin{itemize}
%We have some remarks

\item
The above assumptions
(and additional assumptions provided later)
are not restrictive,
% as they appear,
%look like,
and are easily verifiable.
%applicable.
%The evidence is that they can be
Such conditions are
 widely used
 %satisfied
 in
%almost
%most
popular models
 %in biology, ecology, chemistry, economics, etc
 in the literature; see Section \ref{sec:app}.

\item
Parts (2) and (3) of Assumption \ref{asp1} guarantee the existence and uniqueness of a strong solution to \eqref{Kol-Eq}.
We need part (1) of Assumption \ref{asp1} to ensure that the solution to \eqref{Kol-Eq} is a non-degenerate diffusion.
Moreover, as will be seen
later that (3) implies the tightness of the family of transition probabilities associated with
 the solution to \eqref{Kol-Eq}.
 One difficulty stems from the positive term $A_2\int_{-r}^0h\big(\Vphi(s)\big)\mu(ds)$ on the right-hand side of \eqref{asp1-3}, which cannot be relaxed in practice.
\item
Assumption \ref{asp-2} plays an important role in guaranteeing the $\pi$-uniform integrability of the function $\sum_i \big(\abs{f_i(\cdot)}+g_i^2(\cdot)\big)$, for any invariant measure $\pi$.
It
 will become clear
in Lemma \ref{lem-integrable} and Lemma \ref{weak-converge} as well as the remaining parts of the paper.
\end{itemize}
\end{rem}

%\section{Main Results}\label{sec:res}
As was alluded to, persistence and extinction are concepts of vital importance in biology and ecology. It turns out that such concepts are features shared by all stochastic functional \ko systems.
While
the termination of a species in biology is referred to as extinction, the moment of extinction is generally considered to be the death of the last individual of the species.
In contrast
 to extinction, we have the persistence of a species.
 To proceed,
similar to \cite{HN18, SS12, SBA11},
we define persistence and extinction as follows.

\begin{deff} \label{def-persist} {\rm
%Recall that
Let $\vec X(t)= (X_1(t),\dots,X_n(t))^\top$ be the solution of \eqref{Kol-Eq}.
	The process $\vec X$ is strongly stochastically persistent if for any $\eps>0$, there exists an $R=R(\eps)>0$ such that for any $\Bphi\in \C_+^\circ$
	\begin{equation}%\label{et1.0}
	\liminf_{t\to\infty}\PPphi\left\{R^{-1}\leq \abs{X_i(t)}\leq R\right\}\geq 1-\eps\text{ for all }i=1,\dots,n.
	\end{equation}
}\end{deff}

\begin{deff} {\rm With $\vec X(t)$ given in Definition \ref{def-persist}, for  $\Bphi\in\C_+^\circ $
%$i=1,dots,n$,
and some  $i \in \{1,\dots,n\}$, we say
$X_i$ goes extinct with probability $p_{\Bphi}>0$ if
	\[
	\PPphi\left\{\lim_{t\to\infty}X_i(t)=0\right\}=p_{\Bphi}.
	\] }
\end{deff}

{\color{blue}
For simplicity, we will sometime use ``persistent" for ``strongly stochastically persistent" and ``persistence" for ``strongly stochastic persistence"; and use these terminologies exchangeably.} 

Let $\M$ be the set of ergodic invariant probability measures of $\vec X_t$ supported on the boundary $\partial \C_+$. Note that if we let $\bdelta^*$ be the Dirac measure concentrated at $\vec 0$, then $\bdelta^*\in\M$ so that $\M\neq\emptyset$. For a subset $\wdt\M\subset \M$, denote by $\conv(\wdt\M)$ the convex hull of $\wdt\M$,
that is, the set of probability measures $\pi$ of the form
$\pi(\cdot)=\sum_{\nu\in\wdt\M}p_\nu\nu(\cdot)$
with $p_\nu\geq 0$  and $\sum_{\nu\in\wdt\M}p_\nu=1$.

\begin{asp}\label{asp-lambda>0}
{\rm
For any $\pi\in\conv(\M)$, we have
$$\max_{i=1,\dots,n}\left\{\lambda_i(\pi)\right\}>0,$$
where
\begin{equation}\label{eq-r-lamba}
\lambda_i(\pi):=\int_{\partial \C_+}\left(f_i(\Vphi)-
\frac{\sigma_{ii}g_i^2(\Vphi)}{2}\right)\pi(d\Vphi).
\end{equation}
}
\end{asp}

\begin{thm}\label{p-theorem1}
Assume that Assumptions {\rm\ref{asp1}},
{\rm\ref{asp-2}}, and
{\rm\ref{asp-lambda>0}} hold.
{\color{blue} 
The solution $\vec X$ of \eqref{Kol-Eq} is strongly stochastically persistent.
}
\end{thm}

It is well-recognized that the nondegeneracy of the diffusion is not sufficient to  imply the strong Feller property as well as the uniqueness of an invariant probability measure of stochastic delay systems. The following assumption is needed to obtain the uniqueness of an invariant probability measure.

\begin{asp}\label{asp-unique-ipm}{\rm
The following conditions hold:
\begin{itemize}
\item [\rm(i)]
There are some constants $D_0,d_0>0$ such that for any $\Vphi^{(1)},\Vphi^{(2)}\in \C_+^\circ$, $i\in \{1,\dots,n\}$,
% then
\begin{equation}\label{u-eq0}
\begin{aligned}
|f_i(\Vphi^{(1)})&-f_i(\Vphi^{(2)})|\leq D_0\abs{\vec x^{(1)}-\vec x^{(2)}}\abs{1+\vec x^{(1)}+\vec x^{(2)}}^{d_0}\\
&+D_0\int_{-r}^0 \abs{\Vphi^{(1)}(s)-\Vphi^{(2)}(s)}\abs{1+\Vphi^{(1)}(s)
+\Vphi^{(2)}(s)}^{d_0}\mu(ds),
\end{aligned}
\end{equation}
 where $\vec x^{(1)}:=\Vphi^{(1)}(0), \vec x^{(2)}:=\Vphi^{(2)}(0)$.
\item [\rm(ii)] The conditions in (i) above holds
%above assumption holds
with $f_i(\cdot)$ replaced by $g_i(\cdot)$  and
%or
$g_i^2(\cdot)$.
\item [\rm(iii)]
%For any $\Vphi\in\C_+^\circ$,
The inverse of matrix $(g_i(\Vphi)g_j(\Vphi)\sigma_{ij})_{n\times n}$ is uniformly bounded in $\C_+^\circ$.
\end{itemize}
}\end{asp}

\begin{prop}\label{thm-unique-ipm}
Under Assumptions {\rm\ref{asp1}} and {\rm\ref{asp-unique-ipm}},  the solution of equation \eqref{Kol-Eq} has at most one invariant probability measure on $\C_+^\circ$.
\end{prop}

\begin{thm}\label{thm3.3}
Under Assumptions {\rm\ref{asp1}}, {\rm\ref{asp-2}}, {\rm\ref{asp-lambda>0}}, and {\rm\ref{asp-unique-ipm}},  system \eqref{Kol-Eq} has a unique invariant probability measure concentrated on $\C_+^\circ$.
\end{thm}

\begin{rem}	
		Assumption \ref{asp-lambda>0} means that all invariant measures on the boundary are repellers (because the maximum Lyapunov exponent of an invariant measure is positive), which guarantees that the solution in the interior cannot stay long near the boundary. As a result, the species coexist.

\end{rem}

\section{Preliminaries and Key
Technical Results}\label{sec:key}
\subsection{Existence, uniqueness, positivity, and key estimates of the solutions}

To begin, we state the functional It\^o formula for our processes; see \cite{CF13} for more details.
Let $\DD$ be the space of c\`adl\`ag functions $\Vphi:[-r,0]\mapsto\R^n$.
For $\Vphi\in\DD$, with $s\ge 0$ and $y\in \R^n$,
we define horizontal and vertical perturbations as
$$
\Vphi_s(t)=
\begin{cases}\Vphi(t+s)\, \text{ if }\, t\in[-r,-s],\\
 \Vphi(0) \, \text{ if }\,t\in[-s,0],
\end{cases}
$$
and
$$
\Vphi^y(t)=
\begin{cases}\Vphi(t)\, \text{ if }\, t\in[-r,0),\\
 \Vphi(0)+y\text{ if }t=0,
\end{cases}
$$
respectively.
The horizontal  and vertical partial derivatives of $V:\DD\to\R$ at $\Vphi$, denoted by $\partial_t V(\Vphi)$, $(\partial_i V(\Vphi))_{i=1}^n$, are defined as
\begin{equation*}%\label{e.dt}
\partial_t V(\Vphi)=\lim\limits_{s\to0} \dfrac{V(\Vphi_s)-V(\Vphi)}s,
\end{equation*}
and
\begin{equation}\label{e.dx}
\partial_i V(\Vphi)=\lim\limits_{s\to0} \dfrac{V(\Vphi^{se_i})-V(\Vphi)}s, i=1,\dots,n,
\end{equation}
respectively, if the limits exist.
In \eqref{e.dx}, $e_i$ is the standard unit vector in $\R^n$
 whose $i$-th component is $1$ and all other components are $0.$
Let $\BF$ be the family of functions $V(\cdot):
\DD\mapsto\R$
satisfying that
\begin{itemize}
  \item $V$ is continuous, that is, for any $\eps>0$, $\Vphi\in \DD$
      there is a $\delta>0$ such that
  $|V(\Vphi)-V(\Vphi')|<\eps$ as long as $\|\Vphi-\Vphi'\|< \delta;$
\item the derivatives
$V_t$, $V_x=(\partial_i V)$, and $V_{xx}=(\partial_{ij} V)$ exist and are continuous;
\item $V$, $V_t$, $V_x=(\partial_i V)$ and $V_{xx}=(\partial_{ij} V)$ are bounded in each set $\{\Vphi\in\DD: \|\Vphi\|\leq R\}$, $R>0$.
\end{itemize}
Let $V(\cdot)\in\BF$, we define the operator
\begin{equation}
\begin{aligned}
\Lom V(\Vphi)
=\partial_t V(\Vphi)+
\sum_{i=1}^n
\varphi_i(0)f_i(\Vphi)\partial_i V(\Vphi)
+\dfrac12
\sum_{i,j=1}^n
\varphi_i(0)\varphi_j(0)\sigma_{ij}g_i(\Vphi)g_j(\Vphi)\partial_{ij}V(\Vphi).
\end{aligned}
\end{equation}
We have the functional It\^o formula (see \cite{CF10, CF13}) as follows
\begin{equation}\label{f.Ito}
\begin{aligned}
d V(\vec X_{t})=\big(\Lom V(\vec X_t)\big)dt
&+\sum_{i=1}^n X_i(t)g_i(\vec X_t)\partial_i V(\vec X_t)dE_i(t).
\end{aligned}
\end{equation}

\begin{lm}\label{LV}
For any $\gamma<\gamma_b$ and $p_0>0$, $\Vrho=(\rho_1,\dots,\rho_n)\in\R^n$ satisfying
\begin{equation}\label{cd-p0}
\abs{\Vrho}<\min\left\{\dfrac{\gamma_b}2, \dfrac1n, \frac{\gamma_b}{4\sigma^*}\right\}\;
\text{and}\;p_0<\min\left\{1,\dfrac{\gamma_b}{8n\sigma^*}\right\},
\end{equation}
where $\sigma^*:=\max\{\sigma_{ij}: 1\leq i,j\leq n\}$,
let
$$\displaystyle V_{\Vrho}(\Vphi):=\Big(1+\vec c^\top\vec x\Big)\prod_{i=1}^nx_i^{\rho_i}\exp\left\{A_2\int_{-r}^0\mu(ds)\int_s^0 e^{\gamma(u-s)}h\big(\Vphi(u)\big)du\right\}.$$
Then,
we have
\begin{equation}\label{est-LV}
\begin{aligned}
\Lom V_{\Vrho}^{p_0}(\Vphi)\leq &p_0V_{\Vrho}^{p_0}(\Vphi)\Bigg[A_0 \1_{\{|\vec x|<M\}}-\gamma_0-Ah(\vec x)
\\&-A_2\gamma \int_{-r}^0\mu(ds)\int_s^0e^{\gamma(u-s)}h\big(\Vphi(u)\big)du-\dfrac{\gamma_b}2\sum_{i=1}^n\Big(\abs{f_i(\Vphi)}+g_i^2(\Vphi)\Big)
\Bigg],
\end{aligned}
\end{equation}
where $\vec x:=\Vphi(0)$ and $A$ is a positive number satisfying $A<A_1-A_2\int_{-r}^0 e^{-\gamma s}\mu(ds)$.
Recall that $\vec c$, $M$, $A_0$, $A_1$, $A_2$, $\gamma_0$, $\gamma_b$, $h(\cdot)$, and $\mu(\cdot)$ are defined in Assumption {\rm\ref{asp1}(3)}.
\end{lm}

\begin{proof}
Let
\begin{equation*}
\begin{aligned}
U_{\Vrho}(\Vphi)&=\ln V_{\Vrho}(\Vphi)
\\&=\ln \Big(1+\vec c^\top\vec x\Big)+\sum_{i=1}^n \rho_i\ln x_i+A_2\int_{-r}^0\mu(ds)\int_s^0 e^{\gamma(u-s)}h\big(\Vphi(u)\big)du.
\end{aligned}
\end{equation*}
By \cite[Remark 2.2]{DY19} and direct calculation, we have
\begin{equation*}
\begin{aligned}
\partial_t U_{\Vrho}(\Vphi)=&A_2h(\vec x)\int_{-r}^0 e^{-\gamma s}\mu(ds)\\
&-A_2\int_{-r}^0h\big(\Vphi(s)\big)\mu(ds)-A_2\gamma \int_{-r}^0\mu(ds)\int_s^0e^{\gamma(u-s)}h\big(\Vphi(u)\big)du,
\end{aligned}
\end{equation*}
\begin{equation*}
\begin{aligned}
\partial_i U_{\Vrho}(\Vphi)=\dfrac{c_i}{1+\vec c^\top\vec x}+\dfrac{\rho_i}{x_i}\;;\;\partial_{ij} U_{\Vrho}(\Vphi)=\dfrac{-c_ic_j}{\Big(1+\vec c^\top\vec x\Big)^2}+\dfrac{-\delta_{ij}\rho_i}{x_i^2},
\end{aligned}
\end{equation*}
where
\[
\delta_{ij}=\begin{cases}
1\quad \text{if}\quad i=j,\\
0\quad\text{otherwise.}
\end{cases}
\]
As a consequence, we obtain from the functional It\^o formula
that
\begin{equation}\label{est-LU}
\begin{aligned}
\Lom U_{\Vrho}(\Vphi)
&=A_2 h(\vec x)\int_{-r}^0 e^{-\gamma s}\mu(ds)-A_2\int_{-r}^0h\big(\Vphi(s)\big)\mu(ds)-A_2\gamma \int_{-r}^0\mu(ds)\int_s^0e^{\gamma(u-s)}h\big(\Vphi(u)\big)du
\\&\quad +\dfrac{\sum_{i=1}^n c_ix_if_i(\Vphi)}{1+\vec c^\top\vec x}-\dfrac 12\sum_{i,j=1}^n  \dfrac{c_ic_j\sigma_{ij}x_ix_jg_i(\Vphi)g_j(\Vphi)}{\Big(1+\vec c^\top\vec x\Big)^2} +\sum_{i=1}^n\rho_i\Big(f_i(\Vphi)-\sigma_{ii}g_i^2(\Vphi)\Big).
\end{aligned}
\end{equation}
Therefore, by the fact $V_{\Vrho}^{p_0}(\Vphi)=e^{p_0U_{\Vrho}(\Vphi)}$ and an application of the
functional
It\^o formula, we get
\begin{equation*}
\begin{aligned}
\Lom V_{\Vrho}^{p_0}(\Vphi)=&p_0V_{\Vrho}^{p_0}(\Vphi)\Bigg(A_2 h(\vec x)\int_{-r}^0 e^{-\gamma s}\mu(ds)-A_2\int_{-r}^0h\big(\Vphi(s)\big)\mu(ds)
\\&\;-A_2\gamma \int_{-r}^0\mu(ds)\int_s^0e^{\gamma(u-s)}h\big(\Vphi(u)\big)du
\\&\;+\dfrac{\sum_{i=1}^n c_ix_if_i(\Vphi)}{1+\vec c^\top\vec x}-\dfrac 12\sum_{i,j=1}^n  \dfrac{c_ic_j\sigma_{ij}x_ix_jg_i(\Vphi)g_j(\Vphi)}{\Big(1+\vec c^\top\vec x\Big)^2}
\\&\;+\sum_{i=1}^n\rho_i\Big(f_i(\Vphi)-\sigma_{ii}g_i^2(\Vphi)\Big)
\\&\;+\dfrac 12 p_0\sum_{i,j=1}^n\Big(\dfrac{c_ix_i}{1+\vec c^\top\vec x}+\rho_i\Big)\Big(\dfrac{c_jx_j}{1+\vec c^\top\vec x}+\rho_j\Big)\sigma_{ij}g_i(\Vphi)g_j(\Vphi)\Bigg).
\end{aligned}
\end{equation*}
Since
\begin{align*}
\dfrac 12 p_0\sum_{i,j=1}^n&\Big(\dfrac{c_ix_i}{1+\vec c^\top\vec x}+\rho_i\Big)\Big(\dfrac{c_jx_j}{1+\vec c^\top\vec x}+\rho_j\Big)\sigma_{ij}g_i(\Vphi)g_j(\Vphi)
\\&\leq \dfrac 14p_0\sum_{i,j=1}^n(1+\rho_i)(1+\rho_j)\sigma_{ij}\Big(g_i^2(\Vphi)+g_j^2(\Vphi)\Big)
\\&\leq 2p_0n\sigma^*\sum_{i=1}^n g_i^2(\Vphi),
\end{align*}
and $\abs{\rho_i}<\frac{\gamma_b}2\;;\;\abs{\rho_i}\sigma^*+2p_0n\sigma^*<\frac{\gamma_b}2\;\forall i=1,\dots,n,$ using Assumption \ref{asp1}, we have
\begin{equation*}
\begin{aligned}
\Lom V_{\Vrho}^{p_0}(\Vphi)\leq& p_0V_{\Vrho}^{p_0}(\Vphi)\Bigg(A_0 \1_{\{\abs{\vec x}<M\}}-\gamma_0-h(\vec x)\Big(A_1-A_2\int_{-r}^0 e^{-\gamma s}\mu(ds)\Big)
\\&-A_2 \gamma \int_{-r}^0\mu(ds)\int_s^0e^{\gamma(u-s)}h\big(\Vphi(u)\big)du-\frac{\gamma_b}2\sum_{i=1}^n\Big(\abs{f_i(\Vphi)}+g_i^2(\Vphi)\Big)\Bigg)
\\\leq& p_0V_{\Vrho}^{p_0}(\Vphi)\Bigg(A_0 \1_{\{\abs{\vec x}<M\}}-\gamma_0-Ah(\vec x)
\\&-A_2\gamma \int_{-r}^0\mu(ds)\int_s^0e^{\gamma(u-s)}h\big(\Vphi(u)
\big)du-\dfrac{\gamma_b}2\sum_{i=1}^n\Big(\abs{f_i(\Vphi)}
+g_i^2(\Vphi)\Big)\Bigg).
\end{aligned}
\end{equation*}
The proof is complete.
\end{proof}

\begin{thm}\label{existence}
For any initial condition $\Bphi\in\C_+$, there exists a unique global solution of \eqref{Kol-Eq}. It remains in $\C_+$ $($resp., $\C_+^{\circ})$, provided $\Bphi\in\C_+$ $($resp., $\Bphi\in\C_+^{\circ})$.
Moreover, for any $p_0,\Vrho$ satisfying
 %the
 condition \eqref{cd-p0}, we have
\begin{equation}\label{EV-bounded-ne}
\Ephi V_{\Vrho}^{p_0}(\vec X_t)\leq V_{\Vrho}^{p_0}(\Bphi)e^{A_0p_0t}.
\end{equation}
In addition, if $\rho_i\geq 0,\forall i$, then
\begin{equation}\label{EV-bounded}
\Ephi V_{\Vrho}^{p_0}(\vec X_t)\leq V_{\Vrho}^{p_0}(\Bphi)e^{-\gamma_0p_0t}+\bar M_{p_0,\Vrho},
\end{equation}
where $$\bar M_{p_0,\Vrho}:=\dfrac {A_0}{\gamma_0}\sup_{\Vphi\in \C_{V,M}}V_{\Vrho}^{p_0}(\Vphi)<\infty\text{ provided }\rho_i\geq 0\;\forall i,$$
and
$\C_{V, M}=\{\Vphi\in\C_+: A_2\gamma \int_{-r}^0\mu(ds)\int_s^0e^{\gamma(u-s)}h\big(\Vphi(u)\big)du\leq A_0 \text{ and } |\vec x|\leq M\}$.
\end{thm}

\begin{proof}
We carry out
the proof for the existence and uniqueness
of the solution with initial value $\Bphi\in \C^{\circ}_+$. The other cases can be handled similarly.
Let $\Vrho^{(1)}=(\rho_1^{(1)},\dots,\rho_n^{(1)})\in\R^n$ with $\rho_i^{(1)}<0\;\forall i=1,\dots,n$ satisfy the conditions \eqref{cd-p0}. We define the following stopping times $\tau_k^{(1)}=\inf\{t\geq 0: V_{\Vrho^{(1)}}^{p_0}(\vec X_t)\geq k\}$
%$$\tau_k=\inf\Big\{t\geq 0: X_i(t)\vee X_i^{-1}(t)\geq k\;\text{for some}\; i \in \{1,2,...,n\}\Big\},$$
and $\tau_\infty^{(1)}=\lim_{k\to\infty}\tau_k^{(1)}.$ It is easily seen that
\begin{equation}\label{exist-0}
\lim_{m\to\infty}\inf\left\{V_{\Vrho^{(1)}}^{p_0}(\Vphi): x_i\vee x_i^{-1}>m\;\text{for some}\;i\in\{1,\dots,n\}, \vec x:=\Vphi(0),\Vphi\in\C_+^\circ\right\}=\infty.
\end{equation}
The existence and uniqueness of local solutions can be seen in \cite{MAO1} due to the local
Lipschitz continuity of the coefficients. To prove the solution is global and remains in $\C_+^{\circ}$, since \eqref{exist-0}, it is sufficient to prove that $\tau_\infty^{(1)}=\infty$\a.s
We obtain from \eqref{est-LV} that
$$\Lom V_{\Vrho^{(1)}}^{p_0}(\Vphi)\leq A_0p_0V_{\Vrho^{(1)}}^{p_0}(\Vphi),\;\forall \Vphi\in \C_+^\circ.$$
Hence, by the functional It\^o formula, we get
\begin{align*}
\Ephi V_{\Vrho^{(1)}}^{p_0}(\vec X_{t\wedge \tau_k^{(1)}})&=V_{\Vrho^{(1)}}^{p_0}(\Bphi)+\Ephi\int_0^{t\wedge \tau_k^{(1)}}\Lom V_{\Vrho^{(1)}}^{p_0}(X_s)ds
\\&\leq V_{\Vrho^{(1)}}^{p_0}(\Bphi)+p_0A_0\int_0^t \Ephi V_{\Vrho^{(1)}}^{p_0}(\vec X_{s\wedge \tau_k^{(1)}})ds.
\end{align*}
Combined with Gronwall's inequality, this implies that
\begin{equation}\label{exist-1}
\Ephi V_{\Vrho^{(1)}}^{p_0}(\vec X_{t\wedge\tau_k^{(1)}})\leq V_{\Vrho^{(1)}}^{p_0}(\Bphi)e^{p_0A_0t},\;\forall t\geq 0.
\end{equation}
As a consequence,
$$\PPphi\Big\{V_{\Vrho^{(1)}}^{p_0}(\vec X_{t\wedge\tau_k^{(1)}})\geq k\Big\}\leq \dfrac{V_{\Vrho^{(1)}}^{p_0}(\Bphi)e^{p_0A_0t}}{k}\to 0\;\text{as}\;k\to\infty,$$
which forces $\tau_\infty^{(1)}>t\a.s$ for any $t>0$ and hence, $\tau_\infty^{(1)}=\infty\a.s$

Now, we consider the second part.  For any $p_0,\Vrho$ satisfying \eqref{cd-p0}, by applying \eqref{est-LV}, one has
$$
\Lom V_{\Vrho}^{p_0}(\Vphi)\leq A_0p_0 V_{\Vrho}^{p_0}(\Vphi)\;\text{for all }\Vphi\in\C_+^\circ.
$$
Thus, from \eqref{exist-1}, we get
$$\Ephi V_{\Vrho}^{p_0}(\vec X_t)\leq V_{\Vrho}^{p_0}(\Bphi)e^{A_0p_0t}.$$
If $\rho_i\geq 0\;\forall i$, a consequence of \eqref{est-LV} is \begin{equation}\label{exists-1111}
\Lom V_{\Vrho}^{p_0}(\Vphi)\leq  \gamma_0p_0\bar M_{p_0,\Vrho}-\gamma_0p_0V_{\Vrho}^{p_0}(\Vphi).
\end{equation}
In \eqref{exists-1111}, we have used the fact
$$A_0 \1_{\{\abs{\vec x}<M\}}-A_2\gamma \int_{-r}^0\mu(ds)\int_s^0e^{\gamma(u-s)}h\big(\Vphi(u)\big)du\leq 0,\text{ if }\Vphi\notin\C_{V,M}.$$
By a standard argument (see e.g., \cite[Theorem 5.2, p.157]{MAO}), we can obtain \eqref{EV-bounded} from \eqref{exists-1111}.
The proof is complete.
\end{proof}

\begin{lm}\label{X-bounded}
For any $R_1>0$, $T>r$, and $\eps>0$, there exists an $R_2>0$ such that
\begin{equation*}
\PPphi\Big\{\norm{\vec X_t}\leq R_2,\;\forall t\in [r,T]\Big\}>1-\eps,
\end{equation*}
for any initial point $\Bphi$ satisfying $V_{\vec 0}(\Bphi)<R_1$, where $V_{\vec 0}$ is defined as in Lemma {\rm\ref{LV}} corresponding to $\Vrho=\vec 0=(0,\dots,0).$
\end{lm}

\begin{proof}
As the proof of Theorem \ref{existence}, we define the following stopping times $\tau_k^{(2)}=\inf\{t\geq 0: V_{\vec 0}^{p_0}(\vec X_t)\geq k\}$ and $\tau_\infty^{(2)}=\lim_{k\to\infty}\tau_k^{(2)}.$
Analogous  \eqref{exist-1}, we obtain
$$\Ephi V_{\vec 0}^{p_0}(\vec X_{t\wedge\tau_k^{(2)}})\leq V_{\vec 0}^{p_0}(\Bphi)e^{p_0A_0t},\;\forall t\geq 0.$$
Therefore, for any $R_1$, $T$, $\eps>0$, and initial condition $\Bphi$ satisfying $V_{\vec 0}(\Bphi)<R_1$, there exists
a finite constant $k_0$ such that
$$\dfrac{V_{\vec 0}^{p_0}(\Bphi)e^{p_0A_0T}}{k_0}< \eps,$$
and
%then
\begin{align*}
\PPphi\Big\{V_{\vec 0}^{p_0}(\vec X_{T\wedge \tau_{k_0}^{(2)}})\geq k_0\Big\}\leq \dfrac{V_{\vec 0}^{p_0}(\Bphi)e^{p_0A_0T}}{k_0}< \eps.
\end{align*}
That means $\PPphi\{\tau_{k_0}^{(2)}\geq T\}> 1-\eps$ or
$$\PPphi\big\{V_{\vec 0}^{p_0}(\vec X_t)\leq k_0\;\forall t\in [0,T]\big\}> 1-\eps.$$
Note that $V_{\vec 0}^{p_0}(\vec X_t)\geq 1+\sum_{i=1}^nc_iX_i(t)$ and $c_i>0\;\forall i=1,\dots,n$.
Therefore, it is easily seen that there exists a finite constant $R_2$ satisfying
$$\PPphi\Big\{\norm{\vec X_t}\leq R_2,\;\forall t\in [r,T]\Big\}>1-\eps.$$
\end{proof}

\begin{lm}\label{Holder}
There is a
sufficiently small $\alpha>0$
such that
for any $R>0$ and $\eps>0$, there exists $R_3=R_3(R,\eps)>0$ satisfying
\begin{equation}
\text {if }\norm{\Bphi}\leq R\text{ then }\PPphi\left\{\|\vec X_{t}\|_{2\alpha}\leq R_3\;\forall t\in[r,3r]\right\}\geq1-\frac{\eps}2.
\end{equation}
As a consequence,
for any $R>0$ and $\eps>0$, there exists an $R_4=R_4(\eps,R)>0$  satisfying that
\begin{equation}
\text {if }V_{\vec 0}(\Bphi)\leq R\text{ then }\PPphi\left\{\|\vec X_{t}\|_{2\alpha}\leq R_4\;\forall t\in[2r,3r]\right\}\geq1-\eps.
\end{equation}
\end{lm}

\begin{proof}
For any $R$ and $\eps>0$,
by slightly modifying the proof of Lemma \ref{X-bounded}, there exists an $\tilde R>0$ depending only on $R$ such that
\begin{equation}\label{h-e1}
\PPphi\{\norm{\vec X_{t}}\leq \tilde R,\text{ for all }t\in[0,3r]\}\geq 1-\frac{\eps}4\text{ if }\norm{\Bphi}\leq R.
\end{equation}
Denote by $f_i^{\tilde R}(\cdot)$ and $g_i^{\tilde R}(\cdot)$ the truncated functions, where
\[f_i^{\tilde R}(\Vphi)=\begin{cases}
f_i(\Vphi)\text{ if } \norm{\Vphi}<\tilde R,\\
f_i\left(\dfrac{R_1\Vphi}{\norm{\Vphi}}\right)\text{ otherwise,}
\end{cases}
\]
and $g_i^{\tilde R}(\cdot)$ is defined similarly.
Then $f_i^{\tilde R}(\cdot)$ and $g_i^{\tilde R}(\cdot)$ are globally Lipschitz and bounded.  Let $\wdt {\vec X}(t)=\big(\wdt X_1(t),\dots,\wdt X_n(t)\big)$ be the solution of \eqref{Kol-Eq} when we replace $f_i(\cdot)$ and $g_i(\cdot)$ by $f_i^{\tilde R}(\cdot)$ and $g_i^{\tilde R}(\cdot)$, respectively.
%Therefore,
By a standard argument, it is easy to obtain that
$$\Ephi \abs{\wdt{X}_i(t)}^4\leq \tilde K\;\forall 0\leq t\leq 3r, \|\Bphi\|\leq R$$
where $\tilde K$ is
%some
a constant depending only on $R$ and $\tilde R$. On the other hand,
by Burkholder's inequality we have that $\forall 0\leq s\leq t\leq 3r, \|\Bphi\|\leq R,$
$$\Ephi\abs{\wdt X_i(t)-\wdt X_i(s)}^4\leq C_{1} \Ephi\left |\int_s^t \wdt X_i(y)dy\right|^4+C_{1}\Ephi\left(\int_s^t |\wdt X_i(y)|^2dy\right)^2,$$
where $C_{1}$ depends only on $T$, $R$, and $\tilde R$.
Hence, by H\"older's inequality, we obtain for $0\leq s\leq t\leq 3r, \|\Bphi\|\leq R$ that
\begin{equation*}
\begin{aligned}
\Ephi\abs{\wdt X_i(t)-\wdt X_i(s)}^4\leq 2C_{1}(t-s)^2\Ephi\int_0^s\left|\wdt X_i(y)\right|^4dy
\leq C_2(t-s)^2,
\end{aligned}
\end{equation*}
where $C_2$ is
a constant depending only on $R$ and $\tilde R$. As a consequence of the Kolmogorov-Chentsov theorem, $\{\wdt {\vec X}(t):0\leq t\leq 3r\}$ has H\"older-continuous sample paths with % some
 an exponent $2\alpha\in(0,\frac12)$. Moreover, there is a $R_3=R_3( R, \eps)$ satisfying
\begin{equation*}%\label{h-e3}
\PPphi\left\{\sup_{0\leq t\leq 3r}|\wdt {\vec X}(t)|+\sup_{0\leq s\leq t\leq 3r}\frac{|\wdt {\vec X}(t)-\wdt {\vec X}(s)|}{(t-s)^{2\alpha}}\leq R_3\right\}\geq 1-\frac\eps4, \|\Bphi\|\leq R,
\end{equation*}
which implies
\begin{equation}\label{h-e3}
\PPphi\left\{\|\wdt {\vec X}_t\|_{2\alpha}\leq R_3\;\forall t\in [r,3r] \right\}\geq 1-\frac\eps4, \|\Bphi\|\leq R.
\end{equation}
Combining \eqref{h-e1} and \eqref{h-e3} implies that
$$\PPphi\left\{\| \vec X_t\|_{2\alpha}\leq R_3\;\forall t\in[r,3r]\right\}\geq 1-\frac{\eps}2,\text{ provided }\norm{\Bphi}<R,$$
and the first part of the proposition is proved.

Now, we consider the second part. By Lemma \ref{X-bounded}, there is an $R_5=R_5(\eps,R)$ such that
\begin{equation}\label{h-e4}
\PPphi\{\norm{\vec X_t}<R_5\;\forall t\in[r,3r]\}\geq 1-\frac \eps 2\text{ if }V_{\vec 0}(\Bphi)<R.
\end{equation}
Hence, the second conclusion follows from the first part, \eqref{h-e4}, and the Markov property of $(\vec X_t)$.
\end{proof}

\begin{prop}\label{theorem-2.4}
The following results hold.
\begin{itemize}
\item [\rm(i)]
Let $\rho_1^{(3)}$ be a fixed constant satisfying $0<\rho_1^{(3)}<\min\left\{\frac{\gamma_b}2, \frac1n, \frac{\gamma_b}{4\sigma^*}\right\}$.
For any $T>r$ and $m>0$ there exists
a finite constant $K_{m,T}$ such that
$$\Ephi\norm{ X_{i,t}}^{p_0\rho_1^{(3)}}\leq K_{m,T}\phi_i^{p_0\rho_1^{(3)}}(0),\;\forall t\in [r,T], i=1,\dots,n,$$
given that
$$\abs{\Bphi(0)}+\int_{-r}^0\mu(ds)\int_s^0e^{\gamma(u-s)}h\big(\Bphi(u)\big)du<m,$$
where $\vec X_t=:(X_{1,t},\dots,X_{n,t})$ and $\Bphi=:(\phi_1,\dots,\phi_n)$ is the initial value.
\item[\rm(ii)] For any $T>r$, $\eps>0$, $R>0$, there exists an $\eps_1>0$ such that
\begin{equation}\label{cont-ini}
\PP\left\{\big\|{\vec X^{\Bphi_1}_T-\vec X^{\Bphi_2}_T}\big\|\leq \eps\right\}\geq 1-\eps\text{ whenever }V_{\vec 0}(\Bphi_i)<R, \norm{\Bphi_1-\Bphi_2}\leq \eps_1.
\end{equation}
Moreover, the solution $(\vec X_t)$ has the Feller property in $\C_+$.
\end{itemize}
\end{prop}

\begin{proof}
Let $\Vrho^{(3)}=(\rho_1^{(3)},0,\dots,0)$. Then $\Vrho^{(3)}$ satisfies
%the
condition \eqref{cd-p0}. By the functional It\^o formula, we obtain
\begin{equation}\label{2.3-2}
\begin{aligned}
V^{\frac{p_0}2}_{\Vrho^{(3)}}(\vec X_t)=&V^{\frac{p_0}2}_{\Vrho^{(3)}}(\Bphi)+\int_0^t \Lom V^{\frac{p_0}2}_{\Vrho^{(3)}}(\vec X_s)ds\\
&+\int_0^t\dfrac {p_0}2  V^{\frac{p_0}2}_{\Vrho^{(3)}}(\vec X_s)\sum_{i=1}^n\left(\dfrac{c_iX_i(s)}{1+\sum_{i'=1}^n c_{i'}X_{i'}(s)}+\delta_{1i}\rho_1^{(3)}\right)g_i(\vec X_s)dE_i(s),
\end{aligned}
\end{equation}
where $\delta_{1i}=1$ if $i=1$ and otherwise, $\delta_{1i}=0$.
Therefore, combining with \eqref{est-LV} leads to that
\begin{equation}\label{7-17-1}
\begin{aligned}
V^{\frac{p_0}2}_{\Vrho^{(3)}}(\vec X_t)\leq& V^{\frac{p_0}2}_{\Vrho^{(3)}}(\Bphi)+A_0p_0\int_0^tV^{\frac{p_0}2}_{\Vrho^{(3)}}(\vec X_s)ds\\
&+\int_0^t\dfrac {p_0}2  V^{\frac{p_0}2}_{\Vrho^{(3)}}(\vec X_s)\sum_{i=1}^n\left(\dfrac{c_iX_i(s)}{1+\sum_{i'=1}^n c_{i'}X_{i'}(s)}+\delta_{1i}\rho_1^{(3)}\right)g_i(\vec X_s)dE_i(s).
\end{aligned}
\end{equation}
In the estimates to follow, in fact we need the terms
in \eqref{supV} to
be finite, which can be done by first using estimates for the solution at stopping time $t\wedge\tau_k$
with $\tau_k$
being the first time  such that $|g(\vec X_s)|\vee V_{\Vrho^{(3)}}(\vec X_s)>k$, and
letting $k \to \infty$.
Since it is a standard argument, we omit it
for brevity.
We obtain from \eqref{7-17-1} that
\begin{equation}\label{supV}
\begin{aligned}
&\Ephi\sup_{t\in [0,T]} \Big[V^{\frac{p_0}2}_{\Vrho^{(3)}}(\vec X_t)\Big]^2\leq C^{(1)}V^{p_0}_{\Vrho^{(3)}}(\Bphi)+C^{(1)}\Ephi \int_0^T \sup_{s'\in [0,s]}\Big[V^{\frac{p_0}2}_{\Vrho^{(3)}}(\vec X_{s'})\Big]^2ds
\\& +C^{(1)}\Ephi\sup_{t\in[0,T]}\abs{\int_0^t V^{\frac{p_0}2}_{\Vrho^{(3)}}(\vec X_s)\sum_{i=1}^n\left(\dfrac{c_iX_i(s)}{1+\sum_{i'=1}^n c_{i'}X_{i'}(s)}+\delta_{1i}\rho_1^{(3)}\right)g_i(\vec X_s)dE_i(s)}^2,
\end{aligned}
\end{equation}
where $C^{(1)}$ is a constant, independent of $\Bphi$. The Burkholder-Davis-Gundy inequality and the H\"older inequality imply that
\begin{equation}\label{BDV-1}
\begin{aligned}
\Ephi\sup_{t\in[0,T]}&\abs{\int_0^t V^{\frac{p_0}2}_{\Vrho^{(3)}}(\vec X_s)\sum_{i=1}^n\left(\dfrac{c_iX_i(s)}{1+\sum_{i'=1}^n c_{i'}X_{i'}(s)}+\delta_{1i}\rho_1^{(3)}\right)g_i(\vec X_s)dE_i(s)}^2
\\&
\leq
16n\sigma^*\Ephi\int_0^T V^{p_0}_{\Vrho^{(3)}}(\vec X_s)\sum_{i=1}^ng_i^2(\vec X_s)ds,
\end{aligned}
\end{equation}
for a constant $C_p^{(2)}$, independent of $\Bphi$.
In the display above, we have used
$$
\begin{aligned}
\sum_{i,j=1}^n&\left(\dfrac{c_iX_i(s)}{1+\sum_{i'=1}^n c_{i'}X_{i'}(s)}+1\right)\left(\dfrac{c_jX_j(s)}{1+\sum_{i'=1}^n c_{i'}X_{i'}(s)}+1\right)\sigma_{ij}g_i(\vec X_s)g_j(\vec X_s)\\
&\leq 4n\sigma^*\sum_{i=1}^n g_i^2(\vec X_s).
\end{aligned}
$$
It follows from \eqref{supV} and \eqref{BDV-1} that
\begin{equation}\label{supV-1}
\begin{aligned}
\Ephi\sup_{t\in [0,T]}V^{p_0}_{\Vrho^{(3)}}(\vec X_t)\leq& C^{(1)} V^{p_0}_{\Vrho^{(3)}}(\Bphi)+C^{(1)}\Ephi \int_0^T \sup_{s'\in [0,s]}V^{p_0}_{\Vrho^{(3)}}(\vec X_{s'})ds
\\&+16n\sigma^*C^{(1)}\Ephi\int_0^T V^{p_0}_{\Vrho^{(3)}}(\vec X_s)\sum_{i=1}^ng_i^2(\vec X_s)ds.
\end{aligned}
\end{equation}
On the other hand, by \eqref{EV-bounded-ne},
we get
$$\Ephi V_{\Vrho^{(3)}}^{p_0}(\vec X_{t})\leq V_{\Vrho^{(3)}}^{p_0}(\Bphi)e^{p_0A_0t},\;\forall t\geq 0.$$
Therefore, we obtain from the functional It\^o formula and \eqref{est-LV} that
\begin{equation*}
\begin{aligned}
0&\leq \Ephi V^{p_0}_{\Vrho^{(3)}}(\vec X_T)=V^{p_0}_{\Vrho^{(3)}}(\Bphi)+\Ephi \int_0^T  \Lom V^{p_0}_{\Vrho^{(3)}}(\vec X_s)ds
\\&\leq V^{p_0}_{\Vrho^{(3)}}(\Bphi)+\Ephi\int_0^T \left(p_0A_0V_{\Vrho}^{p_0}(\vec X_s)-\dfrac{\gamma_b}2V_{\Vrho}^{p_0}(\vec X_s)\sum_{i=1}^ng_i^2(\vec X_s)\right)ds
\\&\leq K_{T}^{(1)} V^{p_0}_{\Vrho^{(3)}}(\Bphi)-\dfrac{\gamma_b}2\Ephi\int_0^T  V^{p_0}_{\Vrho^{(3)}}(\vec X_s)\sum_{i=1}^ng_i^2(\vec X_s)ds,
\end{aligned}
\end{equation*}
where $K_{T}^{(1)}$ is a finite constant depending only on $T$.
It follows that
\begin{equation}\label{sum-g}
\Ephi\int_0^T  V^{p_0}_{\Vrho^{(3)}}(\vec X_s)\sum_{i=1}^ng_i^2(\vec X_s)ds\leq K_{T}^{(2)} V^{p_0}_{\Vrho^{(3)}}(\Bphi),\;\text{for some constant}\;K_{T}^{(2)}.
\end{equation}
Combining \eqref{supV-1} and \eqref{sum-g} yields that
\begin{equation}\label{supV-2}
\Ephi\sup_{t\in [0,T]}V^{p_0}_{\Vrho^{(3)}}(\vec X_t)\leq K_{T}^{(3)} V^{p_0}_{\Vrho^{(3)}}(\Bphi)+K_T^{(3)}\Ephi \int_0^T \sup_{s'\in[0,s]}V^{p_0}_{\Vrho^{(3)}}(\vec X_{s'})ds,
\end{equation}
for some constant $K_T^{(3)}$ independent of $\Bphi$.
It is clear that there exists $K_{m,T}^{(4)}$ such that
\begin{equation}\label{supV-1111}
V^{p_0}_{\Vrho^{(3)}}(\Bphi)\leq K_{m,T}^{(4)}\phi_1^{p_0\rho_1^{(3)}}(0),
\end{equation}
given that
$$\abs{\Bphi(0)}+\int_{-r}^0\mu(ds)\int_s^0e^{\gamma(u-s)}h\big(\Bphi(u)\big)du<m.$$
Combining \eqref{supV-2}, \eqref{supV-1111}, and Gronwall's inequality, we have that
\begin{equation}\label{supV-3}
\Ephi\sup_{t\in [0,T]}V^{p_0}_{\Vrho^{(3)}}(\vec X_t)\leq K_{m,T}^{(5)}\phi_1^{p_0\rho_1^{(3)}}(0),
\end{equation}
where $K_{m,T}^{(5)}$ is a finite constant independent of $\Bphi$.
Note that
\begin{equation}\label{supV-4}
V^{p_0}_{\Vrho^{(3)}}(\vec X_t)\geq X^{p_0\rho_1^{(3)}}_{1}(t).
\end{equation}
It follows from \eqref{supV-3} and \eqref{supV-4} that
$$\Ephi\norm{X_{1t}}^{p_0\rho_1^{(3)}}\leq K_{m,T}^{(5)}\phi_1^{p_0\rho_1^{(3)}}(0),\;\forall t\in [r,T].$$
Hence, by a similar argument, we obtain
$$\Ephi\norm{X_{it}}^{p_0\rho_1^{(3)}}\leq K_{m,T}\phi_i^{p_0\rho_1^{(3)}}(0),\;\forall t\in [r,T], i=1,\dots,n,$$
for some constant $K_{m,T}$ depending only on $m,T$. As a result, the first part of the Theorem is proved.

Finally, since our coefficients are Lipschitz continuous in each bounded set of $\C_+$, by using \eqref{EV-bounded-ne} and the truncation argument, the second conclusion is easily obtained (it is similar to the proof of Lemma \ref{Holder}). In addition, the Feller property can be obtained by slightly modifying the proof in  \cite[Lemma 2.9.4 and Theorem 2.9.3]{MAO}.
\end{proof}

\subsection{Tightness, weak convergence of occupation measures, and uniform integrability}
Let $\Vrho=\vec 0$. We obtain from \eqref{est-LV} that for all $\Vphi\in \C_+$, $\vec x:=\Vphi(0)$,
\begin{equation*}
\Lom V_{\vec 0}^{p_0}(\Vphi)\leq \gamma_0p_0\bar M_{p_0,\vec 0}-Ap_0 V_{\vec 0}^{p_0}(\Vphi)h(\vec x),
\end{equation*}
where $\bar M_{p_0,\vec 0}$ is defined as in Theorem \ref{existence}.
Hence, by the functional It\^o formula, we have
\begin{equation*}
\begin{aligned}
\Ephi V_{\vec 0}^{p_0}(\vec X_t)\leq V_{\vec 0}^{p_0}(\Bphi)+\gamma_0p_0\bar M_{p_0,\vec 0}t-\Ephi\int_0^t Ap_0 V_{\vec 0}^{p_0}(\vec X_s)h(\vec X(s))ds.
\end{aligned}
\end{equation*}
Since $V_{\vec 0}(\Vphi)\geq 1+\vec c^\top\vec x$,
we get
\begin{equation}\label{t-eq6}
\int_0^T \Ephi\left(1+\sum_{i=1}^nc_iX_i(t)\right)^{p_0}h(\vec X(t))dt\leq \dfrac 1{Ap_0}\left(V_{\vec 0}^{p_0}(\Bphi)+T\gamma_0p_0\bar M_{p_0,\vec 0}\right),\;\forall T\geq 0.
\end{equation}
A consequence of \eqref{t-eq6} is that
there is a constant $H_1$ such that
\begin{equation}\label{t-eq7}
\begin{aligned}
\int_r^T \Ephi\bigg(&\Big(1+\sum_{i=1}^nc_iX_i(t)\Big)^{p_0}h(\vec X(t))\\
&+\int_{-r}^0 \Big(1+\sum_{i=1}^n c_i X_i(t+s)\Big)^{p_0}h(\vec X(t+s))\mu(ds)\bigg)dt\\
\leq& H_1\left(T+V_{\vec 0}^{p_0}(\Bphi)\right),\;\forall T\geq r.
\end{aligned}
\end{equation}

On the other hand, using \eqref{est-LV} again, we have
\begin{equation}
\Lom V_{\vec 0}^{p_0}(\Vphi)\leq \gamma_0p_0\bar M_{\vec 0,p_0}-\gamma_0p_0 V_{\vec 0}^{p_0}(\Vphi)\text{ for all }\Vphi\in\C_+.
\end{equation}
Therefore, similarly to the process of getting \eqref{t-eq6}, we obtain
\begin{equation}\label{t-eq1}
\int_0^T \Ephi V_{\vec 0}^{p_0}(\vec X_t)dt\leq \frac 1{p_0\gamma_0}(T\gamma_0p_0\bar M_{p_0,\vec 0}+V_{\vec 0}^{p_0}(\Bphi)),\;\forall T\geq 0.
%\text{ provided }V_{\vec 0}^{p_0}(\Bphi)<R.
\end{equation}
Combining \eqref{t-eq1} and the Markov inequality leads to that for any $\eps,R>0$ there exists a finite constant $R_1=R_1(\eps,R)$ such that
\begin{equation}\label{t-eq2}
\dfrac 1t\int_0^t \Ephi\1_{\{V_{\vec 0}^{p_0}(\vec X_s)<R_1\}}ds\geq 1-\frac\eps 2,\text{ provided } V_{\vec 0}(\Bphi)<R.
\end{equation}
Because of \eqref{t-eq2}, Lemma \ref{Holder}, and the Markov property of $\vec X_t$, for any $\eps,R>0$, there exists a compact subset $\K=\K(\eps,R):=\{\Vphi: \|\Vphi\|_{2\alpha}\leq R_4\}$ of $\C_+$ satisfying
\begin{equation}\label{t-eq3}
\dfrac 1t\int_{2r}^{t+2r} \Ephi\1_{\{\vec X_s\in \K\}}ds\geq 1-\eps,\text{ provided } V_{\vec 0}(\Bphi)<R.
\end{equation}
In the above, $R_4=R_4(\eps,R)$ is determined as in Lemma \ref{Holder};  the compactness of $\K$ in $\C$ follows the Sobolev embedding theorem.

For each $t>r$, define the following occupation measures
\begin{equation}\label{8-7-pi}\Pi^{\Bphi}_t(\cdot):=\frac 1t\Ephi\int_r^t \1_{\{\vec X_s\in\cdot\}}ds.\end{equation}
Then it follows from \eqref{t-eq3} that for $V_{\vec 0}(\Bphi)<R$,
\begin{equation}\label{t-eq4}
\left\{\Pi_t^{\Bphi}(\cdot): t\geq 2r\right\}\text{ is tight}.
\end{equation}
Note that $\Pi^{\Bphi}_t(\cdot)$ defined in \eqref{8-7-pi} is  a subprobability measure for each $t>2r$. However, its weak$^*$-limit is still a probability measure.

\begin{lm}\label{lem-integrable}
Under Assumption {\rm\ref{asp-2}(b)},
%one has that
there is a constant, still denoted by $H_1$
%as in \eqref{t-eq7}
(for simplicity of notation)
%$H_1'$ $($still denoted by $H_1$ as in \eqref{t-eq7} for simplicity of notation$)$
such that
\begin{equation}\label{h2-eq0}
\begin{aligned}
\int_r^T \Ephi\bigg(&\Big(1+\sum_{i=1}^nc_iX_i(t)\Big)^{p_0}h_1(\vec X(t))\\
&+\int_{-r}^0 \Big(1+\sum_{i=1}^n c_i X_i(t+s)\Big)^{p_0}h_1(\vec X(t+s))\mu_1(ds)\bigg)dt\\
\leq& H_1\left(T+V_{\vec 0}^{p_0}(\Bphi)\right),\;\forall T\geq r.%\text{ provided }\norm{\Bphi}<R.
\end{aligned}
\end{equation}
\end{lm}

\begin{proof}
By \eqref{est-LV}, we have
\begin{equation*}
\Lom V_{\vec 0}^{p_0}(\Vphi)\leq \gamma_0p_0\bar M_{p_0,\vec 0}-\frac{p_0\gamma_b}2 V_{\vec 0}^{p_0}(\Vphi)\sum_{i=1}^n\Big(\abs{f_i(\Vphi)}+g_i^2(\Vphi)\Big).
\end{equation*}
In view of the functional It\^o formula,
\begin{equation}%\label{h2-eq1}
\begin{aligned}
\Ephi V_{\vec 0}^{p_0}(\vec X_t)\leq& V_{\vec 0}^{p_0}(\Bphi)+\gamma_0p_0\bar M_{p_0,\vec 0}t\\
&-\frac{p_0\gamma_b}2 \Ephi\int_0^t \Big(1+\sum_{i=1}^nc_iX_i(s)\Big)^{p_0}\sum_{i=1}^n\Big(\abs{f_i(\vec X_s)}+g_i^2(\vec X_s)\Big)ds.
\end{aligned}
\end{equation}
Therefore,
we get
\begin{equation}\label{h2-eq1}
\begin{aligned}
\int_0^T \Ephi&\left(1+\sum_{i=1}^nc_iX_i(t)\right)^{p_0}\sum_{i=1}^n\Big(\abs{f_i(\vec X_s)}+g_i^2(\vec X_s)\Big)ds\\
&\leq \dfrac 2{p_0\gamma_b}\left(V_{\vec 0}^{p_0}(\Bphi)+T\gamma_0p_0\bar M_{p_0,\vec 0}\right)\text{ for all }T\geq 0.
\end{aligned}
\end{equation}
In view of \eqref{A2} and \eqref{h2-eq1},
\begin{equation}\label{h2-eq2}
\int_0^T \Ephi\left(1+\sum_{i=1}^nc_iX_i(t)\right)^{p_0}h_1(\vec X_s)ds\leq \dfrac 2{b_1p_0\gamma_b}\left(V_{\vec 0}^{p_0}(\Bphi)+T\gamma_0p_0\bar M_{p_0,\vec 0}\right),\;\forall T\geq 0.
\end{equation}
Hence, we obtain \eqref{h2-eq0}.
\end{proof}

\begin{rem}
It is easily seen that $\sum_i \abs{f_i(\Vphi)}+g_i^2(\Vphi)$ is uniformly integrable owing to either \eqref{t-eq7} and Assumption \ref{asp-2}(a) or \eqref{h2-eq0} and Assumption \ref{asp-2}(b).
 Lemma \ref{lem-integrable} reveals that
 %the
 Assumption \ref{asp-2}(b) can play the same role as
 %the
 Assumption \ref{asp-2}(a) in guaranteeing the uniform integrability of $\sum_i \abs{f_i(\Vphi)}+g_i^2(\Vphi)$. Hence, from now on, when we assume
 %the
 Assumption \ref{asp-2} holds, without loss of generality, we can assume that
 %the
 Assumption \ref{asp-2}(a) holds.
\end{rem}

\begin{lm}\label{weak-converge}
Assume that $(\Bphi_k)_{k\in\N}\subset \C_+$, $(T_k)_{k\in\N}\subset\R_+$ are such that $V_{\vec 0}(\Bphi_k)\leq R$, $T_k> r$, $\lim_{k\to\infty}T_k=\infty$,  and the sequence $(\Pi_{T_k}^{\Bphi_k})_{k\in\N}$ converges weakly to a probability measure $\pi$. Then $\pi$ is an invariant probability measure and moreover
%one has
\begin{equation}\label{convergence-G}
\lim_{k\to\infty}\int_\C G(\Vphi)\Pi_{T_k}^{\Bphi_k}(d\Vphi)=\int_\C G(\Vphi)\pi(d\Vphi),
\end{equation}
for any function $G:\C_+\to\R$ satisfying
\begin{equation}\label{in-condition}
\abs{G(\Vphi)}\leq K_G\left((1+\vec c^\top\vec x)^ph(\vec x)+\int_{-r}^0 \Big(1+\sum_{i=1}^n c_i\varphi_i(s)\Big)^ph(\Vphi(s))\mu(ds)\right),
\end{equation}
for some $p<p_0$, $\text{ where } \vec x:=\Vphi(0)$.
%{\color{red}
Likewise, if Assumption {\rm\ref{asp-2} (b)} holds, we also have \eqref{convergence-G} for $G$ satisfying
$$\abs{G(\Vphi)}\leq K_G\left((1+\vec c^\top\vec x)^ph_1(\vec x)+\int_{-r}^0 \Big(1+\sum_{i=1}^n c_i\varphi_i(s)\Big)^ph_1(\Vphi(s))\mu_1(ds)\right),
$$
$\text{ where } \vec x:=\Vphi(0).$
\end{lm}

\begin{proof}
For the proof of
$\pi$
being an invariant probability measure, we refer to \cite[Theorem 9.9]{EK09}, or \cite [Proposition 6.4]{EHS15} with a slight modification. We proceed to prove the second conclusion.
For any $\eps>0$,
let $l_\eps$ be sufficiently large such that for any $\Vphi$ satisfying $|\vec x|+\int_{-r}^0|\Vphi(s)|\mu(ds)\geq 2l_\eps$,
\begin{equation}\label{t-eq8}
\frac{(1+\vec c^\top\vec x)^ph(\vec x)+\int_{-r}^0  \Big(1+\sum_{i=1}^n c_i\varphi_i(s)\Big)^ph(\Vphi(s))\mu(ds)}{(1+\vec c^\top\vec x)^{p_0}h(\vec x)+\int_{-r}^0  \Big(1+\sum_{i=1}^n c_i\varphi_i(s)\Big)^{p_0}h(\Vphi(s))\mu(ds)}\leq \frac{\eps}{K_GH_1(1+R^{p_0})}.
\end{equation}
The above inequality follows from the fact that
$$\lim_{|\vec x|\to\infty}\frac{(1+\vec c^\top\vec x)^ph(\vec x)}{(1+\vec c^\top\vec x)^{p_0}h(\vec x)}=0$$
and
$$\lim_{\int_{-r}^0|\Vphi(s)|\mu(ds)\to\infty}\frac{\int_{-r}^0  \left(1+\sum_{i=1}^n c_i\varphi_i(s)\right)^ph(\Vphi(s))\mu(ds)}{
\int_{-r}^0  \left(1+\sum_{i=1}^n c_i\varphi_i(s)\right)^{p_0}h(\Vphi(s))\mu(ds)}=0 \,\,\text{ (because }h(\cdot)\geq1).$$

Denote by $u_{l_\eps}:\C\to[0,1]$
a continuous function satisfying
\[
u_{l_\eps}(\Vphi)=
\begin{cases}
1\text { if } |\vec x|+\int_{-r}^0|\Vphi(s)|\mu(ds)\leq 2l_\eps,\\
0\text{ if } |\vec x|+\int_{-r}^0|\Vphi(s)|\mu(ds)\geq 4l_\eps.
\end{cases}
\]
By Tonelli's theorem, we get that
\begin{equation}\label{t-eq9}
\begin{aligned}
\int_\C&\left((1+\vec c^\top\vec x)^{p_0}h(\vec x)+\int_{-r}^0 \Big(1+\sum_{i=1}^n c_i\varphi_i(s)\Big)^{p_0}h(\Vphi(s))\mu(ds)\right)\Pi_{T_k}^{\Bphi_k}(d\Vphi)\\
=&\dfrac 1{T_k}\int_r^{T_k}\E_{\Bphi_k}\bigg(\Big(1+\sum_{i=1}^nc_iX_i(t)\Big)^{p_0}h(\vec X(t))\\
&\hspace{3cm}+\int_{-r}^0 \Big(1+\sum_{i=1}^n c_i X_i(t+s)\Big)h(\vec X(t+s))\mu(ds)\bigg)dt.
\end{aligned}
\end{equation}
Because of \eqref{in-condition}, \eqref{t-eq8}, \eqref{t-eq9}, and \eqref{t-eq7}, one gets
\begin{equation}\label{t-eq10}
\begin{aligned}
\int_\C &\left(1-u_{l_\eps}(\Vphi)\right)|G(\Vphi)|\Pi_{T_k}^{\Bphi_k}(d\Vphi)\\
\leq & K_G\int_\C\left(1-u_{l_\eps}(\Vphi)\right)\bigg((1+\vec c^\top\vec x)^ph(\vec x)\\
&\hspace{3cm}+\int_{-r}^0  \Big(1+\sum_{i=1}^n c_i\varphi_i(s)\Big)^ph(\Vphi(s))\mu(ds)\bigg)\Pi_{T_k}^{\Bphi_k}(d\Vphi)\\
\leq &\frac{\eps}{H_1(1+R^{p_0})}\int_\C\left(1-u_{l_\eps}(\Vphi)\right)\bigg((1+\vec c^\top\vec x)^{p_0}h(\vec x)\\
&\hspace{3cm}+\int_{-r}^0 \Big(1+\sum_{i=1}^n c_i\varphi_i(s)\Big)^{p_0}h(\Vphi(s))\mu(ds)\bigg)\Pi_{T_k}^{\Bphi_k}(d\Vphi)\\
\leq &\eps.
\end{aligned}
\end{equation}
Similarly, because of \eqref{t-eq7} and  $\pi$ being invariant, we have
\begin{equation}\label{t-eq11}
\begin{aligned}
\int_\C \left(1-u_{l_\eps}(\Vphi)\right)|G(\Vphi)|\pi(d\Vphi)
\leq \eps.
\end{aligned}
\end{equation}
The weak convergence of $\Pi_{T_k}^{\Bphi_k}$ to $\pi$ implies 
\begin{equation}\label{t-eq12}
\begin{aligned}
\lim_{k\to\infty}\int_\C u_{l_\eps}(\Vphi)|G(\Vphi)|\Pi_{T_k}^{\Bphi_k}(d\Vphi)=\int_\C u_{l_\eps}(\Vphi)|G(\Vphi)|\pi(d\Vphi).
\end{aligned}
\end{equation}
Combining \eqref{t-eq10}, \eqref{t-eq11}, and \eqref{t-eq12} yields that
\begin{equation*}
\begin{aligned}
\limsup_{k\to\infty}\abs{\int_\C |G(\Vphi)|\Pi_{T_k}^{\Bphi_k}(d\Vphi)-\int_\C |G(\Vphi)|\pi(d\Vphi)}
\leq 2\eps.
\end{aligned}
\end{equation*}
Hence, the proof of the lemma is concluded by letting $\eps\to0$.
\end{proof}

\begin{lm}\label{log-Lap}
Let $Y$ be a random variable, $\theta_0>0$ be a constant, and suppose $$\E \exp(\theta_0 Y)+\E \exp(-\theta_0 Y)\leq K_1$$
for some finite constant $K_1$.
Then the log-Laplace transform
$\eta(\theta)=\ln\E\exp(\theta Y)$
is twice differentiable on $\left[0,\frac{\theta_0}2\right)$ and
$$\dfrac{d\eta}{d\theta}(0)= \E Y,$$
$$0\leq \dfrac{d^2\eta}{d\theta^2}(\theta)\leq K_2\,, \theta\in\left[0,\frac{\theta_0}2\right),$$
 for some $K_2>0$ depending only on $K_1$.
\end{lm}
{\color{blue}
\begin{proof}
	The proof of this lemma can be found in \cite[Proof of Lemma 3.5]{HN18}.
\end{proof}	
}
\section{Persistence}\label{sec:per}
This section is devoted to proving Theorem \ref{p-theorem1} and Theorem \ref{thm3.3}.
It is shown in \cite[Lemma 4]{SBA11}, by the min-max
principle that Assumption \ref{asp-lambda>0} is equivalent to the existence of $\Vrho^*=(\rho^*_1,\dots,\rho^*_n)$ with $\rho^*_i>0$ such
that
\begin{equation}\label{p-eq1}
\inf_{\pi\in\M}\left\{\sum_{i=1}^n \rho^*_i\lambda_i(\pi)\right\}:=2\kappa^*>0.
\end{equation}
By rescaling if necessary, we can assume that $\abs{\Vrho^*}$ is sufficiently small to satisfy condition \eqref{cd-p0}.

\begin{lm}\label{lem-q0}
Assume Assumptions {\rm\ref{asp1}} and {\rm\ref{asp-2}} hold.
For any invariant measure $\pi$, one has
\begin{equation*}
\int_{\C_+} Q_{\vec 0}(\Vphi)\pi(d\Vphi)=0,
\end{equation*}
where
\begin{equation*}
\begin{aligned}
Q_{\vec 0}(\Vphi)=&A_2 h(\vec x)\int_{-r}^0 e^{-\gamma s}\mu(ds)-A_2\int_{-r}^0h\big(\Vphi(s)\big)\mu(ds)
\\&-A_2\gamma \int_{-r}^0\mu(ds)\int_s^0e^{\gamma(u-s)}h\big(\Vphi(u)\big)du
\\&+\dfrac{\sum_{i=1}^n c_ix_if_i(\Vphi)}{1+\vec c^\top\vec x}-\dfrac 12\sum_{i,j=1}^n  \dfrac{c_ic_j\sigma_{ij}x_ix_jg_i(\Vphi)g_j(\Vphi)}{\Big(1+\vec c^\top\vec x\Big)^2}.
\end{aligned}
\end{equation*}
\end{lm}

\begin{proof}
Because of \eqref{t-eq7}, \eqref{h2-eq0}, Lemma \ref{weak-converge}, and Assumption \ref{asp-2}, $Q_{\vec 0}$ is $\pi$-integrable. By the strong law of large numbers (see, e.g., \cite[Theorem 4.2]{Kha12}) we have
\begin{equation}\label{q0-e1}
\lim_{t\to\infty}\frac 1t\int_0^t Q_{\vec 0}(\vec X_s)ds=\int_{\C_+}Q_{\vec 0}(\Vphi)\pi(d\Vphi),\quad\PP_\pi\text{-}\a.s.,
\end{equation}
and
$$
\begin{aligned}
\lim\limits_{t\to\infty}&\dfrac1t\int_0^t \dfrac{\sum_{i,j}\sigma_{ij} c_ic_jX_i(s)X_j(s)g_i(\vec X_s)g_j(\vec X_s)}{(1+\sum_ic_iX_i(s))^2}ds\\
&=\int_{\C_+}\dfrac{\sum_{i,j} \sigma_{ij}c_ic_jx_ix_jg_i(\Vphi)g_j(\Vphi)}{(1+\vec c^\top\vec x)^2}\pi(d\Vphi)<\infty\,~~~~\PP_{\pi}\text{-}\a.s, \text{ where }\vec x:=\Vphi(0).
\end{aligned}
$$
The above limit tells us that if we let
$\langle L_\cdot,L_\cdot \rangle_t$
be the quadratic variation of the local martingale
\[
L_t:= \int_0^t\dfrac{\sum_ic_iX_i(s)g_i(\vec X_s)dE_i(s)}{1+\sum_ic_iX_i(s)},
\]
then
\[
\limsup_{t\to\infty}\frac{\langle L_\cdot,L_\cdot \rangle_t}{t} = \int_{\C_+}\dfrac{\sum_{i,j} \sigma_{ij}c_ic_jx_ix_jg_i(\Vphi)g_j(\Vphi)}{(1+\vec c^\top\vec x)^2}\pi(d\Vphi)<\infty\,~~~~\PP_{\pi}\text{-a.s.}
\]
Applying the strong law of large numbers for local martingales  (see \cite[Theorem 1.3.4]{MAO}), 
\begin{equation}\label{q0-e2}
\lim\limits_{t\to\infty}\dfrac1t\int_0^t \dfrac{\sum_ic_iX_i(s)g_i(\vec X_s)dE_i(s)}{1+\sum_ic_iX_i(s)}=0\,\,\,\,\PP_{\pi}\text{-}\text{a.s.}
\end{equation}
As in \eqref{est-LU}, we have $\Lom U_{\vec 0}(\Vphi)=Q_{\vec 0}(\Vphi),$ where
\begin{equation}\label{q0-e3}
U_{\vec 0}(\Vphi)=\ln (1+\vec c^\top \vec x)+A_2\int_{-r}^0\mu(ds)\int_s^0 e^{\gamma(u-s)}h\big(\Vphi(u)\big)du,\vec x:=\Vphi(0).
\end{equation}
Combining \eqref{q0-e1}, \eqref{q0-e2}, \eqref{q0-e3}, and the functional It\^o formula yields that
\begin{equation}\label{q0-e4}
0\leq \lim\limits_{t\to\infty}\dfrac{U_{\vec 0}(\vec X_t)}t=\int_{\C_+}Q_{\vec 0} (\Vphi)\pi(d\Vphi)\,\,\,\,\PP_{\pi}\text{-}\text{a.s.}
\end{equation}
A simple contradiction argument coupled with \eqref{q0-e4} and \eqref{EV-bounded} leads to that
\begin{equation*}
\int_{\C_+} Q_{\vec 0}(\Vphi)\pi(d\Vphi)=0.
\end{equation*}
\end{proof}

\begin{lm}\label{p-l1}
Assume  Assumptions {\rm\ref{asp1}}, {\rm\ref{asp-2}}, and {\rm\ref{asp-lambda>0}} hold.
Let $\Vrho^*$ be as in \eqref{p-eq1}.
For any compact set $\K$ of $\C_+$, there exists a $T_\K>r$ such that for any $T\geq T_\K$ and $\Bphi\in\partial\C_+\cap\K$, we have
\begin{equation}\label{p-l1-eq0}
\int_r^T \Ephi Q_{\Vrho^*}(\vec X_t)dt\leq -\kappa^*T,
\end{equation}
where
$$Q_{\Vrho^*}(\Vphi):=Q_{\vec 0}(\Vphi)-\sum_{i=1}^n\rho^*_i\left(f_i(\Vphi)-\frac {\sigma_{ii}g_i^2(\Vphi)}2\right).$$
\end{lm}

\begin{proof}
We prove the lemma by using a contradiction argument. Suppose that we can find $\Bphi_k\in \partial\C_+\cap\K$ and $T_k>r$, $T_k\uparrow\infty$ such that
\begin{equation}\label{p-l1-eq0-8-8}
\int_r^{T_k} \Ephi Q_{\Vrho^*}(\vec X_t)dt\geq -\kappa^*T_k.
\end{equation}
Since $\Ephi |Q_{\Vrho^*}(\vec X_t)|\leq H_1\wdt K(t+V_{\vec 0}^{p_0}(\Bphi))$ due to \eqref{t-eq7}, \eqref{h2-eq0}, and Assumption \ref{asp-2}, we can apply
 Tonelli's theorem to obtain 
\begin{equation*}
\int_{\C_+} Q_{\Vrho^*}(\Vphi)\Pi_{T_k}^{\Bphi_k}(d\Vphi)=\frac 1{T_k}\int_r^{T_k} \E_{\Bphi_k}Q_{\Vrho^*}(\vec X_t)dt.
\end{equation*}
Under Assumption \ref{asp-2}, as a consequence of Lemma \ref{weak-converge}, 
\begin{equation}\label{p-l1-eq1}
\lim_{k\to\infty}\frac 1{T_k}\int_r^{T_k}\E_{\Bphi_k}Q_{\Vrho^*}(\vec X_t)dt=\int_{\C_+} Q_{\Vrho^*}(\Vphi)\pi(d\Vphi),
\end{equation}
where the invariant measure $\pi$ is the weak limit of $\{\Pi_{T_k}^{\Bphi_k}\}$. Since the initial values lie on the boundary, $\pi$ is supported in $\partial \C_+$. This combined with Lemma \ref{lem-q0} implies that
\begin{equation}\label{p-l1-eq2}
\int_{\C_+} Q_{\Vrho^*}(\Vphi)\pi(d\Vphi)=-\sum_{i=1}^n \rho^*_i \lambda_i(\pi).
\end{equation}
Thus, we obtain from \eqref{p-eq1}, \eqref{p-l1-eq1}, and \eqref{p-l1-eq2} that
\begin{equation}\label{p-l1-eq3}
\lim_{k\to\infty}\frac 1{T_k}\int_r^{T_k}\E_{\Bphi_k}Q_{\Vrho^*}(\vec X_t)dt\leq -2\kappa^*.
\end{equation}
Combining \eqref{p-l1-eq0-8-8} and \eqref{p-l1-eq3} leads to a contradiction. As a result, the Lemma is proved.
\end{proof}

Now, let $n^*$ be sufficiently large to satisfy
\begin{equation}\label{e:n*}
\gamma_0 (n^*-1)- A_0>0,
\end{equation}
and $p_1>p_0$ but %to still satisfy 
\eqref{cd-p0} still holds.
Under Assumption \ref{asp-2}, a consequence of \eqref{t-eq7} is that there is an $H_1^*$ satisfying
\begin{equation}\label{H1*}
\int_r^T \Ephi Q_{\Vrho^*}(\vec X_t)dt\leq H_1^*(T+V_{\vec 0}^{p_0}(\Bphi))\text{ for all }T\geq r.
\end{equation}
Because of \eqref{EV-bounded}, we have
\begin{equation}\label{EV-bounded-l2}
\Ephi V_{\vec 0}^{p_1}(\vec X_t)\leq V_{\vec 0}^{p_1}(\Bphi)e^{-\gamma_0p_1t}+\bar M_{p_1,\vec 0}.
\end{equation}
Note that
\begin{equation}\label{add-CVM}
\text{if } \Bphi\in \C_{V,M} \text{ then }V_{\vec 0}(\Bphi)\leq (1+M\abs{\vec c})e^{A_0},
\end{equation}
where $\C_{V,M}$ is defined as in Theorem \ref{existence}.
Equations
\eqref{EV-bounded-l2} and \eqref{add-CVM} imply that
\begin{equation}\label{p-l2-eq10}
\Ephi V_{\vec 0}^{p_1}(\vec X_t)\leq R_{V,M} \text{ if }\Bphi\in\C_{V,M}, t\geq 0
\end{equation}
for some $R_{V,M}>0$.
%Moreover, since \eqref{EV-bounded-l2} we can let $R_{V,M}$ be sufficiently large so that
%\begin{equation}\label{RVM-11}
%\Ephi V_{\vec 0}^{p_1}(\vec X_{2r})\leq R_{V,M} \text{ if }V_{\vec 0}^{p_1}(\Bphi)<R_{V,M}.
%\end{equation}
Let $\eps^*\in(0,\frac1{9})$ be such that
\begin{equation}\label{p-l2-eq8}
R_{V,M}^{\frac{p_0}{p_1}}\left(\eps^*\right)^{\frac{p_1-p_0}{p_1}} +\eps^* H_1^* T\leq \frac{\kappa^*}{10}(T-2r)\text{ for any }T\geq 3r.
\end{equation}
To take care of the case when $X_i(t)$ is small but the norm of the segment function $\norm{X_{t}}$ is not, we 
%have
derive the following Lemma.

\begin{lm}\label{p-l2}
Assume  Assumptions {\rm\ref{asp1}}, {\rm\ref{asp-2}}, and {\rm\ref{asp-lambda>0}} hold.
There exist a $\hat\delta>0$ and $T^*>0$ such that for any $T\in[T^*, n^*T^*]$,
\begin{equation}\label{p-p1-eq0}
\Ephi\int_{0}^{T}Q_{\Vrho^*}(\vec X_t)dt\leq -\frac12\kappa^* T, \text { if }\Bphi\in \C_V(\hat\delta),
\end{equation}
where  $$\C_V(\hat\delta):=\left\{\Bphi\in\C^{\circ}_+\cap\C_{V,M} \text{ such that }|\phi_i(0)|\leq\hat\delta \text{ for some }i\right\}.$$
\end{lm}

\begin{proof}
For any event $\A$ with $\PPphi(\A)\geq 1-\eps$, we obtain from \eqref{H1*}, the H\"older inequality, and \eqref{p-l2-eq10} that
\begin{equation}\label{p-l2-eq5}
\begin{aligned}
\Ephi\1_{\A^c}\int_{3r}^{T}Q_{\Vrho^*}(\vec X_t)dt\leq& \E\1_{\A^c}H_1^*\left(T+V_{\vec 0}^{p_0}(\vec X_{2r})\right)
\\ \leq& \left(\eps^*\right)^{\frac{p_1-p_0}{p_1}}\left(\Ephi V_{\vec 0}^{p_1}(\vec X_{2r})\right)^{\frac{p_0}{p_1}}+\eps^* H_1^* T\\\leq& R_{V,M}^{\frac{p_0}{p_1}}\left(\eps^*\right)^{\frac{p_1-p_0}{p_1}} +\eps^* H_1^* T\text{ if }\Bphi\in\C_{V,M},
%\leq& \frac{\kappa^*}8(T-r),
\end{aligned}
\end{equation}
where $\A^c=\Omega\setminus\A$.
Applying Lemma \ref{Holder} implies that there is a compact set $\K^*=\K^*(\eps^*):=\{\Vphi\in\C_+:\norm{\Vphi}\leq R, \norm{\Vphi}_{2\alpha}-\norm{\Vphi}\leq R_4\}$ ($R_4=R_4(\eps^*)$ is as in Lemma \ref{Holder}) such that
\begin{equation}\label{p-l2-eq1}
\PPphi\left\{\vec X_t\in\K^*\text{ for all }t\in[2r,3r]\right\}\geq 1-\frac{\eps^*}2\text{ if }\Bphi\in\C_{V,M}.
\end{equation}
%Since $\Bphi\in\partial \C_+$, $\K^*\subset\partial\C_+$.
In view of Lemma \ref{p-l1},
there exists $T^*=T^*(\K^*)>0$ such that for all $T\geq T^*-3r$,
\begin{equation}\label{5-16-1111}
\int_r^T \Ephi Q_{\Vrho^*}(\vec X_t)dt\leq -\kappa^*T \text{ if }\Bphi\in\partial \C_+\cap\K^*.
\end{equation}
%due to Lemma \ref{p-l1}.
Without loss of generality, we can choose $T^*>3r$
%be
sufficiently large such that for all $T\geq T^*$,
\begin{equation}\label{p-l2-eq9}
-\frac {7}{10}\kappa^*(T-2r)
%+H_1(3r+R_{V,M})\leq -\frac6{10}\kappa^*T
%\text{ and }
%-\frac 6{10}\kappa^*T+
+3A_0r\leq -\frac12\kappa^* T.
\end{equation}
In view of the Feller property of $\vec X_t$ and \eqref{5-16-1111}, we
%can
obtain that there is a $\delta_1>0$ such that for all $T\in [T^*-3r, n^*T^*]$
\begin{equation}\label{p-l2-eq3}
\int_r^T \Ephi Q_{\Vrho^*}(\vec X_t)dt\leq -\dfrac{9}{10}\kappa^*T\text{ if } \Bphi\in\K^*, \text{dist}(\Bphi,\partial \C_+)<\delta_1.%, T\in [T^*+r,n^*T^*+r].
\end{equation}
By virtue of \eqref{p-l2-eq1}, part (i) of Proposition \ref{theorem-2.4}, and the structure of $\K^*$, there exists a $\hat\delta>0$ such that
\begin{equation}\label{p-l2-eq4}
\PPphi(\A)\geq 1-\eps^*\text{ if } \Bphi\in \C_V(\hat\delta),
\end{equation}
where
$$\A=\left\{\text{dist}(\vec X_{2r},\partial \C_+)<\delta_1,\vec X_{2r}\in\K^*\right\}.$$
Combining \eqref{p-l2-eq3} and \eqref{p-l2-eq4} leads to that
for all $T\in[T^*,n^*T^*]$
\begin{equation}\label{p-l2-eq6}
\Ephi\1_\A\int_{3r}^{T}Q_{\Vrho^*}(\vec X_t)dt\leq -\dfrac{9}{10}\kappa^*(T-2r)(1-\eps^*)\leq -\dfrac{8}{10}\kappa^*(T-2r)\text{ if } \Bphi\in \C_V(\hat\delta) .
\end{equation}
We obtain from \eqref{p-l2-eq5}, \eqref{p-l2-eq8}, and \eqref{p-l2-eq6} that for all $T\in[T^*,n^*T^*]$
\begin{equation}\label{p-l2-eq7}
\Ephi\int_{3r}^{T}Q_{\Vrho^*}(\vec X_t)dt\leq-\dfrac{7}{10}\kappa^*(T-2r), \Bphi\in \C_V(\hat\delta).
\end{equation}
Using the functional It\^o formula, Jensen's inequality, and \eqref{EV-bounded-ne}, we have 
\begin{equation}\label{p-p1-eq0''}
\begin{aligned}
\int_0^{3r}\Ephi Q_{\Vrho^*}(\vec X_s)ds=&\frac 1{p_0}\Ephi \left(\ln V_{-\Vrho^*}^{p_0}(\vec X_{3r})-\ln  V_{-\Vrho^*}^{p_0}(\Bphi)\right)\\
\leq&\frac 1{p_0}\ln \frac{\Ephi V_{-\Vrho^*}^{p_0}(\vec X_{3r})}{V_{-\Vrho^*}^{p_0}(\Bphi)}
\leq \frac {\ln e^{3A_0p_0r}}{p_0}=3A_0r.
\end{aligned}
\end{equation}
%{\color{red}Can we obtain \eqref{p-p1-eq0''} for $\int_0^{3r}$ then combining with \eqref{p-p1-eq7} so that we don't need \eqref{p-p1-eq0'}?}
Therefore, we obtain from \eqref{p-l2-eq7}, \eqref{p-p1-eq0''}, and \eqref{p-l2-eq9} that $\text { if }\Bphi\in \C_V(\hat\delta),$
\begin{equation*}%\label{p-p1-eq0}
\Ephi\int_{0}^{T}Q_{\Vrho^*}(\vec X_t)dt\leq -\dfrac{7}{10}\kappa^*(T-2r)+3A_0r\leq -\frac12\kappa^* T,\text{ for all }T\in[T^*,n^*T^*].
\end{equation*}
 The lemma is proved.
\end{proof}

\begin{prop}\label{p-prop1}
Assume that Assumptions {\rm\ref{asp1}}, {\rm\ref{asp-2}}, and {\rm\ref{asp-lambda>0}} hold.
Then there are $\theta\in\left(0,\frac{p_0}2\right)$ and $\wdt K_\theta>0$ such that for any $T\in[T^*,n^*T^*]$ and $\Bphi\in\C^{\circ}_+\cap\C_{V,M}$,
\begin{equation}
\Ephi V_{-\Vrho^*}^{\theta}(\vec X_T)\leq V_{-\Vrho^*}^{\theta}(\Bphi)\exp\Big(-\frac 14\theta\kappa^*T\Big)+\wdt K_\theta,
\end{equation}
where $-\Vrho^*=\left(-\rho^*_1,\dots,-\rho^*_n\right)$.
\end{prop}

\begin{proof}
%Let $\eps\in(0,\frac17)$ be such that $$R_{V,M}^{\frac{p_0}{p_1}}\eps^{\frac{p_1-p_0}{p_0}} +\eps H_1 T\leq \frac{\kappa^*}8(T+r)\text{ for any }T\geq 2r,$$
%where $R_{V,M}$ is as in \eqref{p-l2-eq10}.
By the functional It\^o formula, we obtain that
\begin{equation}\label{p-p1-eq1}
\begin{aligned}
\ln V_{-\Vrho^*}(\vec X_T)
=&\ln V_{-\Vrho^*}(\Bphi)+\int_0^T Q_{\Vrho^*}(\vec X_t)dt\\
&+\int_0^T \left(\frac{\sum_{i}c_iX_i(t)g(\vec X_t)dE_i(t)}{1+\sum_{i}c_iX_i(t)}-\sum_i \rho^*_ig_i(\vec X_t)dE_i(t)\right)\\
=:&\ln V_{-\Vrho^*}(\Bphi)+z(T).
\end{aligned}
\end{equation}
Because of \eqref{p-p1-eq1} and \eqref{EV-bounded-ne}, we have
\begin{equation}\label{p-p1-eq2}
\Ephi\exp\left(p_0z(T)\right)=\frac{\Ephi V_{-\Vrho^*}^{p_0}(\vec X_T)}{V_{-\Vrho^*}^{p_0}(\Bphi)}\leq e^{A_0p_0T}.
\end{equation}
Another consequence of \eqref{EV-bounded-ne} is that
\begin{equation}\label{p-p1-eq3}
\frac{\Ephi V_{\Vrho^*}^{p_0}(\vec X_T)}{V_{\Vrho^*}^{p_0}(\Bphi)}\leq e^{A_0p_0T}.
\end{equation}
We obtain from the definition of $V_{\Vrho}^{p_0}(\Vphi)$ that
\begin{equation}\label{p-p1-eq4}
\begin{aligned}
V_{-\Vrho^*}^{-p_0}(\Vphi)=&\left(1+\vec c^\top \vec x\right)^{-2p_0}\exp\Big\{-2p_0A_2\int_{-r}^0\mu(ds)\int_s^0 e^{\gamma(u-s)}h\big(\Vphi(u)\big)du\Big\}V_{\Vrho^*}^{p_0}(\Vphi)\\
\leq& V_{\Vrho^*}^{p_0}(\Vphi), \text{ where }\vec x:=\Vphi(0).
\end{aligned}
\end{equation}
Applying \eqref{p-p1-eq4} and \eqref{p-p1-eq3} to \eqref{p-p1-eq1} yields
\begin{equation}\label{p-p1-eq5}
\begin{aligned}
&\Ephi \exp(-p_0 z(T))=\dfrac{\Ephi V_{-\Vrho^*}^{-p_0}(\vec X_T)}{V_{-\Vrho^*}^{-p_0}(\Bphi)}%\leq \dfrac{\Ephi V_{\Vrho^*}^{p_0}(\vec X_T)}{V_{-\Vrho^*}^{-p_0}(\vec x)}\\
\\
\leq& \dfrac{\Ephi V_{\Vrho^*}^{p_0}(\vec X_T)}{V_{\Vrho^*}^{\delta_0}(\Bphi)}\left(1+\vec c^\top \Bphi(0)\right)^{-2p_0}\exp\Big\{-2p_0A_2\int_{-r}^0\mu(ds)\int_s^0 e^{\gamma(u-s)}h\big(\Bphi(u)\big)du\Big\}\\
\leq& \left(1+\vec c^\top \Bphi(0)\right)^{-2p_0}\exp\Big\{-2p_0A_2\int_{-r}^0\mu(ds)\int_s^0 e^{\gamma(u-s)}h\big(\Bphi(u)\big)du\Big\}\exp(A_0p_0T).
\end{aligned}% \text{ for } \theta\in[0,\delta_0].
\end{equation}
In view of \eqref{p-p1-eq2} and \eqref{p-p1-eq5}, an application of Lemma \ref{log-Lap} for $z(T)$ implies that there is $\tilde K_2\geq 0$ such that
$$0\leq \dfrac{d^2\tilde\eta_{\Bphi,T}}{d\theta^2}(\theta)\leq \tilde K_2\,\text{ for all }\,\theta\in\left[0,\frac{p_0}2\right),\, \Bphi\in\C_V(\hat\delta), T\in [T^*,n^*T^*],$$
where
$$\tilde\eta_{\Bphi,T}(\theta)=\ln\Ephi \exp(\theta z(T)).$$
Hence, using Lemma \ref{log-Lap} and \eqref{p-p1-eq0} yields
$$\dfrac{d\tilde\eta_{\Bphi,T}}{d\theta}(0)=\Ephi z(T)\leq -\dfrac12\kappa^*T\,\text{ for }\, \Bphi\in \C_V(\hat\delta), T\in [T^*,n^*T^*].$$
By a Taylor expansion around $\theta=0$, for $\Bphi\in \C_V(\hat\delta), T\in [T^*,n^*T^*]$, and $\theta\in\left[0,\frac{p_0}2\right)$, we have
$$\tilde\eta_{\Bphi,T}(\theta)\leq -\dfrac12\kappa^*T\theta+\theta^2\tilde K_2 .$$
Now, if we choose $\theta\in\left(0,\frac{p_0}2\right)$ satisfying
$\theta<\frac{\kappa^*T^*}{4\tilde K_2}$,  we get
\begin{equation}\label{p-p1-eq6}
\tilde\eta_{\Bphi,T}(\theta)\leq -\dfrac14\kappa^*T\theta\,\,\text{ for all }\,\Bphi\in\C_V(\hat\delta), T\in [T^*,n^*T^*].
\end{equation}
In light of \eqref{p-p1-eq6}, we have for such $\theta$, $\Bphi\in \C_V(\hat\delta)$, and $T\in [T^*,n^*T^*]$ that
\begin{equation}\label{p-p1-eq7}
\dfrac{\Ephi V_{-\Vrho^*}^\theta(\vec X_T)}{V_{-\Vrho^*}^\theta(\Bphi)}=\exp \tilde\eta_{\Bphi,T}(\theta)\leq\exp\left(-\frac{1}{4}\kappa^*T\theta\right).
\end{equation}
On the other hand, because of \eqref{EV-bounded-ne},
we have for $\Bphi\notin\C_V(\hat\delta)$ but satisfying $\Bphi\in \C_+^\circ\cap\C_{V,M}$ and $T\in  [T^*,n^*T^*]$ that
\begin{equation}\label{p-p1-eq8}
\Ephi V_{-\Vrho^*}^\theta(\vec X_T)\leq \exp(\theta n^*T^*H)\sup\limits_{\Bphi\notin \C_V(\hat\delta), \Bphi\in \C_{V,M}}\{V_{-\Vrho^*}^{\theta}(\Bphi)\}=:\wdt K_\theta<\infty.
\end{equation}
Combining \eqref{p-p1-eq7} and \eqref{p-p1-eq8} completes the proof.
\end{proof}

\begin{thm}\label{p-theorem1-111}
Assume that Assumptions {\rm\ref{asp1}}, {\rm\ref{asp-2}}, and  {\rm\ref{asp-lambda>0}} hold.
There is a finite constant $K^*$ such that
\begin{equation*}
\limsup_{t\to\infty}\Ephi V_{-\Vrho^*}^{\theta}(\vec X_t)\leq K^*\text{ for all }\Bphi\in\C^{\circ}_+.
\end{equation*}
As a result, 
{\color{blue} 
	The solution $\vec X$ of \eqref{Kol-Eq} is strongly stochastically persistent.
}

\end{thm}

\begin{proof}
	Once we obtained Proposition \ref{p-prop1}, the proof is similar to \cite[Theorem 4.1]{HN18}.
By virtue of \eqref{est-LV}, we have
\begin{equation}\label{et1.1}
\Lom V_{-\Vrho^*}^\theta(\Vphi)\leq -\theta\gamma_0V_{-\Vrho^*}^\theta(\Vphi) \text{ if }\Vphi\notin \C_{V,M},
\end{equation}
where $\C_{V,M}$ is as in Theorem \ref{existence}.
Define the 
%following 
stopping time
\begin{equation}\label{e:tau}
\tau=\inf\{t\geq0: \vec X_t\in \C_{V,M}\}.
\end{equation}
We
obtain from Dynkin's formula and \eqref{et1.1} that
$$
\begin{aligned}
\Ephi&\left[ \exp\big\{\theta\gamma_0\big(\tau\wedge n^*T^*\big)\big\}V_{-\Vrho^*}^\theta\big(\vec X_{\tau\wedge n^*T^*}\big)\right]\\
&\leq V_{-\Vrho^*}^\theta(\Bphi) +\Ephi \int_0^{\tau\wedge n^*T^*}\exp\{\theta\gamma_0 s\}\left[\Lom V_{-\Vrho^*}^\theta(\vec X_s)+ \theta\gamma_0V_{-\Vrho^*}^\theta(\vec X_s)\right]ds\\
&\leq V_{-\Vrho^*}^\theta(\Bphi).
\end{aligned}
$$
As a consequence,
\begin{equation}\label{et1.2}
\begin{aligned}
V_{-\Vrho^*}^\theta(\Bphi)\geq&
\Ephi\left[ \exp\big\{\theta\gamma_0\big(\tau\wedge n^*T^*\big)\big\}V_{-\Vrho^*}^\theta\big(\vec X_{\tau\wedge n^*T^*}\big)\right]\\
=&
 \Ephi \left[\1_{\{\tau\leq (n^*-1)T^*\}} \exp\big\{\theta\gamma_0\big(\tau\wedge n^*T^*\big)\big\}V_{-\Vrho^*}^\theta\big(\vec X_{\tau\wedge n^*T^*}\big)\right]\\
 &+\Ephi \left[\1_{\{ (n^*-1)T^*<\tau<n^*T^*\}} \exp\big\{\theta\gamma_0\big(\tau\wedge n^*T^*\big)\big\}V_{-\Vrho^*}^\theta\big(\vec X_{\tau\wedge n^*T^*}\big)\right]\\
&+ \Ephi \left[\1_{\{\tau\geq n^*T^*\}}\exp\big\{\theta\gamma_0\big(\tau\wedge n^*T^*\big)\big\}V_{-\Vrho^*}^\theta\big(\vec X_{\tau\wedge n^*T^*}\big)\right]\\
\geq&
 \Ephi \left[\1_{\{\tau\leq (n^*-1)T^*\}}V_{-\Vrho^*}^\theta(\vec X_{\tau})\right]\\
 &+\exp\left\{\theta\gamma_0 ((n^*-1)T^*)\right\}\Ephi \left[\1_{\{ (n^*-1)T^*<\tau<n^*T^*\}}V_{-\Vrho^*}^\theta(\vec X_{\tau})\right]\\
&+\exp\left\{\theta\gamma_0 n^*T^*\right\} \Ephi \left[\1_{\{\tau\geq n^*T^*\}}V_{-\Vrho^*}^\theta(\vec X_{n^*T^*})\right].\\
 \end{aligned}
\end{equation}
Combining the Markov property of $(\vec X_t)$ and
Proposition \ref{p-prop1} yields
\begin{equation}\label{et1.3}
\begin{aligned}
\Ephi&\left[ \1_{\{\tau\leq (n^*-1)T^*\}}V_{-\Vrho^*}^\theta(\vec X_{n^*T^*})\right]\\
&\leq
 \Ephi \left[\1_{\{\tau\leq (n^*-1)T^*\}}\big[\wdt K_\theta+e^{-\frac{1}{4}\theta \kappa^*(n^*T^*-\tau)}V_{-\Vrho^*}^\theta(\vec X_{\tau})\big]\right]\\
 &\leq \wdt K_\theta+ \exp\left(-\frac{1}{4}\theta \kappa^*T^*\right)\Ephi\left[\1_{\{\tau\leq (n^*-1)T^*\}}V_{-\Vrho^*}^\theta(\vec X_{\tau})\right].
 \end{aligned}
\end{equation}
Using again the strong Markov property of $(\vec X_t)$ and
\eqref{EV-bounded-ne}, we obtain that
\begin{equation}\label{et1.4}
\begin{aligned}
\Ephi&\left[ \1_{\{(n^*-1)T^*<\tau<n^*T^*\}}V_{-\Vrho^*}^\theta(\vec X_{n^*T^*})\right]\\
&\leq
 \Ephi \left[\1_{\{(n^*-1)T^*<\tau<n^*T^*\}}e^{\theta A_0(n^*T^*-\tau)}V_{-\Vrho^*}^\theta(\vec X_{\tau})\right]\\
 &\leq \exp(\theta A_0T^*)\Ephi\left[\1_{\{(n^*-1)T^*<\tau<n^*T^*\}}V_{-\Vrho^*}^\theta(\vec X_{\tau})\right].
 \end{aligned}
\end{equation}
Applying \eqref{et1.3} and \eqref{et1.4} to \eqref{et1.2} leads to 
\begin{equation}\label{et1.5}
\begin{aligned}
V_{-\Vrho^*}^\theta(\Bphi)
\geq& \exp\left(\frac{1}{4}\theta \kappa^*T^*\right)\Ephi\left[\1_{\{\tau\leq (n^*-1)T^*\}}V_{-\Vrho^*}^\theta(\vec X_{n^*T^*})\right]-\exp\left(\frac{1}{4}\theta \kappa^*T^*\right)\wdt K_\theta\\
 &+\exp(-\theta A_0T^*)\exp\left(\theta\gamma_0 (n^*-1)T^*\right)\Ephi \left[\1_{\{ (n^*-1)T^*<\tau<n^*T^*\}}V_{-\Vrho^*}^\theta(\vec X_{n^*T^*})\right]\\
&+\exp\left(\theta\gamma_0 n^*T^*\right) \Ephi \left[\1_{\{\tau\geq n^*T^*\}}V_{-\Vrho^*}^\theta(\vec X_{n^*T^*})\right]\\
\geq & \exp( m\theta T^*)\Ephi V_{-\Vrho^*}^\theta(\vec X_{n^*T^*})-\wdt K_\theta\exp\left(\frac{1}{4}\theta \kappa^*T^*\right),
 \end{aligned}
\end{equation}
where $m=\min\left\{\frac{1}{4} \kappa^*, \gamma_0 n^*, \gamma_0 (n^*-1)- A_0\right\}>0$ by \eqref{e:n*}. As a result,
\begin{equation*}
\Ephi V_{-\Vrho^*}^\theta(\vec X_{n^*T^*})\leq \hat q_1 V_{-\Vrho^*}^\theta(\Bphi)+q_1^*\,\text{ for all }\, \Bphi\in\C^{\circ}_+,
\end{equation*}
for some $0<\hat q_1<1$, $0<q_1^*<\infty$. Therefore, by the Markov property of $\vec X_t$, we have
\begin{equation*}
\Ephi V_{-\Vrho^*}^\theta(\vec X_{(k+1)n^*T^*})\leq \hat q_1 \Ephi V_{-\Vrho^*}^\theta(\vec X_{kn^*T^*})+q_1^*\,\text{ for all }\, \Bphi\in\C^{\circ}_+,
\end{equation*}
Using this recursively, we obtain
\begin{equation}\label{et1.6}
\Ephi V_{-\Vrho^*}^\theta (\vec X_{kn^*T^*})\leq \hat q_1^kV_{-\Vrho^*}^\theta(\Bphi)+\dfrac{q_1^*(1-\hat q_1^k)}{1-\hat q_1}.
\end{equation}
We obtain from \eqref{et1.6} and \eqref{EV-bounded-ne} that
\begin{equation*}
\Ephi V_{-\Vrho^*}^\theta (\vec X_{T})\leq \left[\hat q_1^kV_{-\Vrho^*}^\theta(\Bphi)+\dfrac{q_1^*(1-\hat q_1^k)}{1-\hat q_1}\right]e^{A_0\theta T^*} \;\text{for all }T\in[kn^*T^*,kn^*T^*+T^*].
\end{equation*}
Hence, by letting $k\to\infty$, we obtain the existence of a finite constant $K^*$ such that
\begin{equation}\label{et1.7}
\limsup_{t\to\infty}\Ephi V_{-\Vrho^*}^{\theta}(\vec X_t)\leq K^*\text{ for all }\Bphi\in\C^{\circ}_+.
\end{equation}
Finally, 
%\eqref{et1.0} 
the strongly
stochastic persistence of $\vec X$ is obtained by applying Markov's inequality to \eqref{et1.7},
and using
$V_{-\Vrho^*}(\Vphi)\geq \frac{1+\vec c^\top\vec x}{\prod_{i=1}^n x_i^{\rho_i^*}}$ and
$\sum_{i=1}^n\rho_i^*<1$, $\rho^*_i>0$.
\end{proof}

To proceed, we prove the uniqueness of the invariant probability measure under
suitable assumptions.
For $\vec x\in\R^{n,\circ}_+$, $\ln \vec x$ is understood as the component-wise logarithm of $\vec x$.
By using the fact 
$$\abs{\vec x^{(1)}-\vec x^{(2)}}\leq \abs{\ln \vec x^{(1)}-\ln \vec x^{(2)}}\abs{1+\vec x^{(1)}+\vec x^{(2)}}$$ and basic inequalities (Young's inequality and the Cauchy-Schwarz inequality), we obtain from \eqref{u-eq0} and Assumption \ref{asp-unique-ipm} (ii) that there is some constant $D_1$ depending only $D_0,d_0$ satisfying
\begin{equation}\label{u-eq0'-7-20}
\begin{aligned}
&\sum_i\Big(\abs{f_i(\Vphi^{(1)})-f_i(\Vphi^{(2)})}+\abs{g_i(\Vphi^{(1)})-g_i(\Vphi^{(2)})}+\abs{g_i^2(\Vphi^{(1)})-g_i^2(\Vphi^{(2)})}\Big)\\
&\quad\leq D_1\abs{\ln \vec x^{(1)}-\ln \vec x^{(2)}}\abs{1+\vec x^{(1)}+\vec x^{(2)}}^{d_0+1}\\
&\qquad+D_1\int_{-r}^0 \abs{\ln \Vphi^{(1)}(s)-\ln \Vphi^{(2)}(s)}\abs{1+\Vphi^{(1)}(s)+\Vphi^{(2)}(s)}^{d_0+1}\mu(ds),
\end{aligned}
\end{equation}
and
\begin{equation}\label{u-eq0'}
\begin{aligned}
\sum_i&\Big(\abs{f_i(\Vphi^{(1)})-f_i(\Vphi^{(2)})}+\abs{g_i(\Vphi^{(1)})-g_i(\Vphi^{(2)})}\\
&+\abs{g_i^2(\Vphi^{(1)})-g_i^2(\Vphi^{(2)})}\Big)\abs{\ln \vec x^{(1)}-\ln \vec x^{(2)}}\\
\leq& D_1\abs{\ln \vec x^{(1)}-\ln \vec x^{(2)}}^2\abs{1+\vec x^{(1)}+\vec x^{(2)}}^{d_0+1}\\
&\qquad+D_1\int_{-r}^0 \abs{\ln \Vphi^{(1)}(s)-\ln \Vphi^{(2)}(s)}^2\abs{1+\Vphi^{(1)}(s)+\Vphi^{(2)}(s)}^{2d_0+2}\mu(ds),
\end{aligned}
\end{equation}
and
\begin{equation}\label{u-eq0'-7-22}
\begin{aligned}
\sum_i&\Big(\abs{f_i(\Vphi^{(1)})-f_i(\Vphi^{(2)})}+\abs{g_i(\Vphi^{(1)})-g_i(\Vphi^{(2)})}\\
&+\abs{g_i^2(\Vphi^{(1)})-g_i^2(\Vphi^{(2)})}\Big)\abs{\ln \vec x^{(1)}-\ln \vec x^{(2)}}^3\\
\leq& D_1\abs{\ln \vec x^{(1)}-\ln \vec x^{(2)}}^4\abs{1+\vec x^{(1)}+\vec x^{(2)}}^{4d_0+4}\\
&\qquad+D_1\int_{-r}^0 \abs{\ln \Vphi^{(1)}(s)-\ln \Vphi^{(2)}(s)}^4\abs{1+\Vphi^{(1)}(s)+\Vphi^{(2)}(s)}^{4d_0+4}\mu(ds).
\end{aligned}
\end{equation}
%where
% $\ln \vec x:=(\ln x_1,\dots,\ln x_n)$
%for $\vec x=(x_1,\dots,x_n)\in \R^n, x_i>0$.
To apply asymptotic couplings method as well as the theories and results in \cite{Ha09}, we
let $Y_i(t)=\ln X_i(t)$ and
consider the following equations
\begin{equation}\label{Yi}
\begin{cases}
dY_i(t)=\left[f_i(\vec X_t)-\frac{g_i^2(\vec X_t)\sigma_{ii}^2}{2}\right]dt+g_i(\vec X_t)dE_i(t),\;i=1,\dots,n,\\
\vec Y_0=\ln \vec \Bphi, \;\Bphi\in \C_+^\circ,
\end{cases}
\end{equation}
and
\begin{equation}\label{Yi-tilde}
\begin{cases}
d\wdt Y_i(t)=\left[f_i(\wdt{\vec X}_t)-\frac{g_i^2(\wdt{\vec X}_t)\sigma_{ii}^2}{2}\right]dt+\wdt\lambda \left[1+X_i(t)+\wdt X_i(t)\right]^{4d_0+4}\left[Y_i(t)-\wdt Y_i(t)\right]dt \\
\hspace{6cm}+g_i(\wdt{\vec X}_t)dE_i(t),\;i=1,\dots,n\\
\wdt{\vec Y}_0=\ln \wdt{\Bphi},\;\wdt{\Bphi}\neq \Bphi\in\C_+^\circ,
\end{cases}
\end{equation}
where $\wdt\lambda$ is sufficiently large to be determined later;
and $\wdt {\vec X}(t):=\left(e^{\wdt Y_1(t)},\dots,e^{\wdt Y_n(t)}\right)$.
%It is similar to \eqref{u-eq0'}, under Assumption \ref{asp-unique-ipm} (ii) we can obtain a similar estimate for $g_i^2(\cdot)$ and then
Put $\vec Z:=\vec Y- \wdt {\vec Y}$.
Combining the functional It\^o formula and \eqref{u-eq0'-7-20}, \eqref{u-eq0'}, \eqref{u-eq0'-7-22}, one has
\begin{equation}\label{u-eq1}
\begin{aligned}
d\abs{\vec Z(t)}^2\leq &\left[D_2\int_{-r}^0 \abs{1+\wdt {\vec X}(t+s)+\vec X(t+s)}^{4d_0+4}\abs{\vec Z(t+s)}^2\mu(ds)\right]dt\\
&- \left(\wdt\lambda-D_2\right) \left|1+\wdt {\vec X}(t)+\vec X(t)\right|^{4d_0+4}\abs{\vec Z(t)}^2dt\\
&+2\sum_{i}\left(Y_i(t)-\wdt Y_i(t)\right)\left(g_i(\vec X_t)-g_i(\wdt{\vec X}_t)\right)dE_i(t),
\end{aligned}
\end{equation}
and
\begin{equation}\label{u-eq1-7-21}
\begin{aligned}
d\abs{\vec Z(t)}^4\leq &\left[D_2\int_{-r}^0 \abs{1+\wdt {\vec X}(t+s)+\vec X(t+s)}^{4d_0+4}\abs{\vec Z(t+s)}^4\mu(ds)\right]dt\\
&- \left(\wdt\lambda-D_2\right) \left|1+\wdt {\vec X}(t)+\vec X(t)\right|^{4d_0+4}\abs{\vec Z(t)}^4dt\\
&+4\sum_{i}\left(Y_i(t)-\wdt Y_i(t)\right)^3\left(g_i(\vec X_t)-g_i(\wdt{\vec X}_t)\right)dE_i(t),
\end{aligned}
\end{equation}
for some constant $D_2$ depending only $D_0$, $d_0$, $\sigma_{ij}$ and independent of $\vec X_t, \wdt{\vec X}_t$.
%where $M(t)$ is some martingale.
For $\Vphi^{(1)},\Vphi^{(2)}\in\C_+^\circ$, define
\begin{align*}
 \wdt U(\Vphi^{(1)},&\Vphi^{(2)}):=\abs{\ln \vec x^{(1)}-\ln \vec x^{(2)}}^4 \\
&+\frac{D_2+9n\sigma^*D_1}{\gamma}\int_{-r}^0 \mu(ds)\int_s^0 e^{\gamma (u-s)}\abs{1+\Vphi^{(1)}(s)
+\Vphi^{(2)}(s)}^{4d_0+4}\\
&\hspace{3cm}\times\abs{\Vphi^{(1)}(s)-\Vphi^{(2)}(s)}^4du,
\end{align*}
and
\begin{align*}
U(\Vphi^{(1)},&\Vphi^{(2)}):=\abs{\ln \vec x^{(1)}-\ln \vec x^{(2)}}^2 \\
&+\frac{D_2+9n\sigma^*D_1}{\gamma}\int_{-r}^0 \mu(ds)\int_s^0 e^{\gamma (u-s)}\abs{1+\Vphi^{(1)}(s)+\Vphi^{(2)}(s)}^{4d_0+4}\\
&\hspace{3cm}\times
\abs{\Vphi^{(1)}(s)-\Vphi^{(2)}(s)}^2du,
\end{align*}
where $\sigma^*:=\max\{\sigma_{ij}: 1\leq i,j\leq n\}$.
% and let
%$$\hat P(\vec X_t,\wdt{\vec X}_t):=\E P(\vec X_t,\wdt{\vec X}_t).$$
Hence, by direct calculations using the functional It\^o formula, \cite[Remark 2.2]{DY19}  and then applying \eqref{u-eq1}, \eqref{u-eq1-7-21}, it is easily seen that we can choose $\wdt \lambda$ being sufficiently large such that
\begin{equation}\label{u-eq2'}
\begin{aligned}
d\wdt U(\vec X_t,\wdt{\vec X}_t)\leq &-2D_3 \bigg( \wdt U(\vec X_t,\wdt{\vec X}_t)+|1+\vec X(t)+\wdt{
\vec X}(t)|^{4d_0+4}\abs{\vec Z(t)}^4\\
&\hspace{1.5cm}+\int_{-r}^0|1+\vec X(t+s)+\wdt{
\vec X}(t+s)|^{4d_0+4}\abs{\vec Z(t+s)}^4\mu(ds) \bigg)dt\\
&+4\sum_{i} \left(Y_i(t)-\wdt Y_i(t)\right)^3\left(g_i(\vec X_t)-g_i(\wdt{\vec X}_t)\right)dE_i(t),
\end{aligned}
\end{equation}
and
\begin{equation}\label{u-eq2'-7-21}
\begin{aligned}
d U(\vec X_t,\wdt{\vec X}_t)\leq &-2D_3 \bigg( U(\vec X_t,\wdt{\vec X}_t)+|1+\vec X(t)+\wdt{
\vec X}(t)|^{4d_0+4}\abs{\vec Z(t)}^2\\
&+\int_{-r}^0|1+\vec X(t+s)+\wdt{
\vec X}(t+s)|^{4d_0+4}\abs{\vec Z(t+s)}^2\mu(ds) \bigg)dt\\
&+2\sum_{i}\left(Y_i(t)-\wdt Y_i(t)\right)\left(g_i(\vec X_t)-g_i(\wdt{\vec X}_t)\right)dE_i(t),
\end{aligned}
\end{equation}
for some positive constant $D_3>8n\sigma^*D_1$.

Similar to \cite[Theorem 3.1]{Ha09}, let
$$\vec v(t)=\wdt\lambda \left[(g_i(\wdt{\vec X}_t)g_j(\wdt{\vec X}_t)\sigma_{ij})_{n\times n}\right]^{-1}\left[1+|\vec X(t)|+|\wdt {\vec X}(t)|\right]^{4d_0+4}\big(\vec Y(t)-\wdt{\vec Y}(t)\big),$$
where $\left[(g_i(\Vphi)g_j(\Vphi)\sigma_{ij})_{n\times n}\right]^{-1}$ is the inverse matrix of matrix $\left[(g_i(\Vphi)g_j(\Vphi)\sigma_{ij})_{n\times n}\right]$ and for each $\eps>0$
$$\wdt {\tau}_\eps:=\inf\left\{t\geq 0:\int_0^t \abs{\vec v(s)}^2ds\geq \eps^{-1}\|\vec Y_0-\wdt{\vec Y}_0\|^2\right\}.$$

\begin{lm}\label{lem-unique-ipm}
%We have
The following assertions hold:
%properties:
\begin{equation}\label{u-tau}
\lim_{\eps\to 0}\PP\{\wdt{\tau}_\eps=\infty\}=1,
\end{equation}
and
\begin{equation}\label{u-limy}
\lim_{t\to \infty}|\vec Y(t)-\wdt{\vec Y}(t)|=0\a.s
\end{equation}
\end{lm}

\begin{rem}{\rm
In the proof of the Lemma, our purpose is to prove $\lim_{t\to\infty}U(\vec X_t,\wdt {\vec X}_t)=0$\a.s To handle the diffusion part of $U(\vec X_t,\wdt {\vec X}_t)$, we need some helps from $\wdt U(\vec X_t,\wdt {\vec X}_t)$. That is why we introduced both $U(\vec X_t,\wdt {\vec X}_t)$ and $\wdt U(\vec X_t,\wdt {\vec X}_t)$ in the above.}
\end{rem}

\begin{proof}
A consequence of \eqref{u-eq2'} and \eqref{u-eq2'-7-21} is that
\begin{equation*}
d \left[e^{D_3t}\E U(\vec X_t,\wdt{\vec X}_t)\right]\leq -D_3e^{D_3t}\E U(\vec X_t,\wdt{\vec X}_t)dt,
\end{equation*}
and
\begin{equation*}
d \left[e^{D_3t}\E \wdt U(\vec X_t,\wdt{\vec X}_t)\right]\leq -D_3e^{D_3t}\E \wdt U(\vec X_t,\wdt{\vec X}_t)dt,
\end{equation*}
which implies that
\begin{equation}\label{u-eq4}
\lim_{t\to\infty}\E e^{D_3t}U(\vec X_t,\wdt{\vec X}_t)=\lim_{t\to\infty}\E e^{D_3t}\wdt U(\vec X_t,\wdt{\vec X}_t)=0,
\end{equation}
and
\begin{equation}\label{u-eq5}
\E\int_0^\infty e^{D_3t}U(\vec X_t,\wdt{\vec X}_t) dt\leq \frac{U(\Bphi,\wdt{\Bphi})}{D_3}.
\end{equation}
An application of Markov's inequality and \eqref{u-eq5}
imply that
\begin{equation}\label{r-eq-0}
\PP\left(\int_0^\infty e^{D_3t}U(\vec X_t,\wdt{\vec X}_t) dt\leq \frac{U(\Bphi,\wdt{\Bphi})}{D_3\sqrt\eps}\right)\geq 1-\sqrt\eps.
\end{equation}
To proceed, similar to the proof of Proposition 5.1 part (i) (in particular, the process of getting \eqref{supV-2}), we can obtain that for $p<p_0$,
$$\E_{\Bphi}\sup_{t\in[0,1]}V_{\vec 0}^p(\vec X_{t})\leq C_1 V_{\vec 0}^p(\vec X_0),$$
for some constant $C_1$. Then, combining with \eqref{EV-bounded}, we have
%$$
%\E V_{\vec 0}^p(\vec X_{t}) \leq C_4 (1+e^{-C_2 t} V_{\vec 0}^p(\vec X_0));
%$$
%which implies
$$\E_{\Bphi}\sup_{t\in[k,k+1]}V_{\vec 0}^p(\vec X_{t})\leq C_2(1+V_{\vec 0}^p(\vec X_0)),$$
for some constant $C_2$.
%Let $\wdt\gamma>0$ be sufficiently small.
We have for any $C>0$,
\begin{align*}
\PP&\left\{ e^{\frac{-D_3t}2} V_{\vec 0}^{4d_0+4}(\vec X_{t})\geq \frac{C}{\sqrt{\eps}}, \text{ for some } t\in[k,k+1]\right\}\\
\leq& \PP\left\{  V_{\vec 0}^{p}(\vec X_{t})\geq \frac{C e^{D_3pk/(8d_0+8)}}{\eps^{p/(8d_0+8)}}, \text{ for some } t\in[k,k+1]\right\}\\
\leq &C_2(1+V_{\vec 0}^p(\vec X_0)) \frac{\eps^{p/(8d_0+8)}}{C e^{D_3pk/(8d_0+8)}}.
\end{align*}
Thus, one has
\begin{equation}\label{eq-8-31-20}
\begin{aligned}
\PP\left\{ e^{-\frac{D_3}2 t} V_{\vec 0}^{4d_0+4}(\vec X_{t})\leq \frac{C}{\sqrt\eps} \forall t\geq 0\right\}
\geq& 1- \sum_{k=0}^\infty C_2(1+V_{\vec 0}^p(\vec X_0)) \frac{\eps^{p/(8d_0+8)}}{C e^{D_3pk/(8d_0+8)}}\\
\geq& 1- K_{C,V_{\vec 0}^p(\vec X_0)} \eps^{p/(8d_0+8)},
\end{aligned}
\end{equation}
for some finite constant $K_{C,V_{\vec 0}^p(\vec X_0)}$ depending on $C,V_{\vec 0}^p(\vec X_0)$.
Combining \eqref{eq-8-31-20} and \eqref{r-eq-0}, we can obtain
\begin{equation}\label{r-eq-5}
\begin{aligned}
\PP&\left(\int_0^\infty\left[1+|\vec X(t)|+|\wdt {\vec X}(t)|\right]^{4d_0+4}\left|\vec Y(t)-\wdt{\vec Y}(t)\right| dt\leq \frac{C}{\eps}\right)\\
&\geq 1-\sqrt\eps-K_{C,V_{\vec 0}^p(\vec X_0)}\eps^{\frac{p}{8d_0+8}}.
\end{aligned}
\end{equation}
We obtain from the definition of $\wdt{\tau}_\eps$ that
\begin{equation}\label{u-eq9}
\begin{aligned}
\PP\{\wdt{\tau}_\eps=\infty\}\geq \PP\left\{\int_0^\infty \abs{\vec v(s)}^2ds < \eps^{-1}\|\Bphi-\wdt{\Bphi}\|^2\right\}.
%\geq &1-\dfrac{\eps\E \int_0^\infty \abs{\vec v(s)}^2ds}{\|\Bphi-\wdt{\Bphi}\|^2}.
\end{aligned}
\end{equation}
Since \eqref{r-eq-5}, \eqref{u-eq9}, definition of $\vec v(\cdot)$, and
%the
Assumption \eqref{asp-unique-ipm}(iii), we
%can
obtain \eqref{u-tau} by letting $\eps\to 0$.

On the other hand, applying the Burkholder-Davis-Gundy inequality and \eqref{u-eq0'-7-20}, we have
\begin{equation}\label{u-eq10}
\begin{aligned}
\E&\sup_{t\in[0,1]}\sum_{i} \left(Y_i(t)-\wdt Y_i(t)\right)\left(g_i(\vec X_t)-g_i(\wdt{\vec X}_t)\right)dE_i(t)\\
&\leq 4\E\left(n\sigma^*\sum_{i}^n\int_0^1\left(Y_i(t)-\wdt Y_i(t)\right)^2\left(g_i(\vec X_t)-g_i(\wdt{\vec X}_t)\right)^2dt\right)^\frac 12\\
&\leq 4n\sqrt{\sigma^*D_1}\Bigg(\E\int_0^1\bigg(|1+\vec X(t)+\wdt{
\vec X}(t)|^{4d_0+4}\abs{\vec Z(t)}^4\\
&\hspace{2.5cm}+\int_{-r}^0|1+\vec X(t+s)+\wdt{
\vec X}(t+s)|^{4d_0+4}\abs{\vec Z(t+s)}^4\mu(ds) \bigg)dt\Bigg)^{\frac 12}.
\end{aligned}
\end{equation}
We obtain from \eqref{u-eq2'}
and the functional It\^o formula
 that
\begin{equation}\label{u-eq10-1111}
\begin{aligned}
\E\int_0^1&\bigg(|1+\vec X(t)+\wdt{
\vec X}(t)|^{4d_0+4}\abs{\vec Z(t)}^4\\
&\hspace{0.5cm}+\int_{-r}^0|1+\vec X(t+s)+\wdt{
\vec X}(t+s)|^{4d_0+4}\abs{\vec Z(t+s)}^4\mu(ds) \bigg)dt\\
&\leq \frac1{2D_3}\E\wdt U(\Bphi,\wdt \Bphi).
\end{aligned}
\end{equation}
Applying \eqref{u-eq10-1111} to \eqref{u-eq10} yields that
\begin{equation}\label{u-eq10-7-21}
\E\sup_{t\in[0,1]}\sum_{i} \left(g_i(\vec X_t)-g_i(\wdt{\vec X}_t)\right)dE_i(t)\leq 4n\sqrt{\sigma^*D_1}\left(\frac 1{2D_3}\E\wdt U(\Bphi,\wdt\Bphi)\right)^\frac 12.
\end{equation}
Hence, combining \eqref{u-eq2'-7-21} and \eqref{u-eq10-7-21}, by a standard argument, we conclude that
\begin{equation}\label{u-eq11}
\E\sup_{t\in[0,1]}U(\vec X_{t+t_0},\wdt{\vec X}_{t+t_0})\leq D_4\left(\E U(\vec X_{t_0},\wdt{\vec X}_{t_0})+\left(\E \wdt U(\vec X_{t_0},\wdt{\vec X}_{t_0})\right)^\frac 12\right),\;\forall t_0>0,
\end{equation}
for some constant $D_4$, independent of $\vec X_t, \wdt{\vec X}_t$. A consequence of Markov's inequality and \eqref{u-eq11} is that
\begin{equation}\label{u-eq12}
\begin{aligned}
\PP&\left\{\sup_{t\in[n-1,n]}U(\vec X_{t},\wdt{\vec X}_{t})\geq e^{-\frac{D_3n}4}\right\}\\
&\leq D_4e^{\frac{D_3n}4}\left(\E U(\vec X_{n-1},\wdt {\vec X}_{n-1})+\left(\E \wdt U(\vec X_{n-1},\wdt{\vec X}_{n-1})\right)^\frac 12\right).
\end{aligned}
\end{equation}
We obtain from \eqref{u-eq4} and \eqref{u-eq12} that
\begin{equation}\label{u-eq13}
\sum_{n=1}^\infty\PP\left\{\sup_{t\in[n-1,n]}U(\vec X_{t},\wdt{\vec X}_{t})\geq e^{-\frac{D_3n}4}\right\}<\infty.
\end{equation}
It follows from the Borel-Cantelli lemma and \eqref{u-eq13} that $\lim_{t\to\infty}U(\vec X_{t},\wdt{\vec X}_{t})=0\a.s$ and thus we get \eqref{u-limy}.
\end{proof}

Once we have Lemma \ref{lem-unique-ipm}, we can mimic the proof of \cite[Theorem 3.1]{Ha09} to obtain the uniqueness of the invariant probability measure of \eqref{Yi}, which is stated as in the following Proposition.

\begin{prop}\label{thm-unique-ipm-11}
Under Assumptions {\rm\ref{asp1}} and {\rm\ref{asp-unique-ipm}},
the solution process of
 \eqref{Yi} has at most one invariant probability measure, and moreover \eqref{Kol-Eq}
has at most one invariant probability measure concentrated on $\C^\circ_+$.
\end{prop}

By the tightness \eqref{t-eq4} of the occupation measures and 
%in view of 
Theorem \ref{p-theorem1-111}, the existence of invariant probability measure of \eqref{Kol-Eq} concentrated on $\C^\circ_+$ is guaranteed. Combined with Proposition \ref{thm-unique-ipm-11}, we have the following Theorem to end this section.

\begin{thm}%\label{thm3.3}
Under  Assumptions {\rm\ref{asp1}}, {\rm\ref{asp-2}}, {\rm\ref{asp-lambda>0}}, and {\rm\ref{asp-unique-ipm}}, system \eqref{Kol-Eq} has a unique invariant probability measure concentrated on $\C_+^\circ$.
\end{thm}

\section{Applications}\label{sec:app}
This section  presents a number of applications of our main results
%(Theorems \ref{thm3.3}, \ref{thm4.1}, and \ref{thm4.2})
to different models.
We
make use of Theorems \ref{thm3.3}
%, \ref{thm4.1}, and \ref{thm4.2}
together with the following
lemma whose proof can be found in \cite{HN18}, to characterize the persistence.

\begin{lm}\label{lm4.1}
	For any $\pi\in\M$ and $i\in I_\pi$, we have
	$\lambda_i(\pi)=0.$
\end{lm}
Moreover, it is worth noting that these sufficient conditions for persistence
%of the following examples
are sharp and are almost necessary in the sense that if they are not satisfied and critical cases are excluded, the extinction will take place, which will be seen in  part (II) \cite{NNY21-2}.

\subsection{Stochastic delay Lotka-Volterra competitive models}\label{subsec:1}
The Lotka-Volterra model, introduced in \cite{Lot25,Vol26}, is one of the most popular models in mathematical biology and
has been studied extensively in
the literature.
When two or more species live in proximity and share the same basic
resources,
%requirements,
they usually compete for food, habitat, territory, etc., we therefore have the Lotka-Volterra competitive model. To capture many complex properties in real life, other terms (white noises, Markov switching, delayed time, etc) are added to the original system.
%Recently, the
Stochastic delay Lotka-Volterra competitive models
have also been
widely studied;
see, for example, \cite{Mao04,Liu17} and references therein.
This kind model for two species has the form
\begin{equation}\label{5.1-eq0}
\begin{cases}
dX_1(t)=X_1(t)\left(a_1-b_{11}X_1(t)-b_{12}X_2(t)-\hat b_{11}X_1(t-r)-\hat b_{12}X_2(t-r)\right)dt\\
\hspace{9cm}+X_1(t)dE_1(t),\\[1ex]
dX_2(t)=X_2(t)\left(a_2-b_{21}X_1(t)-b_{22}X_2(t)-\hat b_{21}X_1(t-r)-\hat b_{22}X_2(t-r)\right)dt\\
\hspace{9cm}+X_2(t)dE_2(t).
\end{cases}
\end{equation}
Note that in the above
$X_i(t)$ is the size of the species $i$ at time $t$;
$a_i>0$ represents the growth rate of the species $i$;
$b_{ii}>0$ is the intra-specific competition of the $i^{th}$ species;
$b_{ij}\geq 0$, ($i\neq j$) stands for the inter-specific competition;
$\hat b_{ij}> -b_{ii}$ ($i,j=1,2$) (i.e., $\hat b_{ij}$ can be negative);
$r$ is the delay time;
$(E_1(t),E_2(t))^\top=\Gamma^\top\vec B(t)$ with
$\vec B(t)=(B_1(t), B_2(t))^\top$ being a vector of independent standard Brownian motions and
$\Gamma$ being a $2\times 2$ matrix such that
$\Gamma^\top\Gamma=(\sigma_{ij})_{2\times 2}$ is a positive definite matrix.

Before applying
our Theorems, let us verify our Assumptions. First, it is easy to see that there is a sufficiently large  $M_1$ such that
\begin{equation}\label{5.1-eq1}
\dfrac{\sum_{i,j=1}^2\sigma_{ij}x_ix_j}{(1+x_1+x_2)^2}\geq 2\sigma_*\text{ if }\abs{\vec x}>M_1, \vec x:=(x_1,x_2),
\end{equation}
for some $\sigma_*>0$.
There exist $0<b_2^*<b_1^*$ and $M_2>0$ satisfying
\begin{equation}\label{5.1-eq2}
\begin{aligned}
&\dfrac{\sum_{i=1}^2x_i\left(a_i-b_{i1}x_1-b_{i2}x_2-\hat b_{i1}\varphi_1(-r))-\hat b_{i2}\varphi_2(-r)\right)}{1+ x_1+x_2}\\
&<-b_1^*(1+\abs{\vec x})+ b_2^*\abs{\Vphi(-r)},
\end{aligned}
\end{equation}
for all $\Vphi\in\C_+$ satisfying $\abs{\vec x}:=\abs{\Vphi(0)}>M_2$,
and
\begin{equation}\label{5.1-eq2-11}
\begin{aligned}
&\dfrac{\sum_{i=1}^2x_i\left(a_i-b_{i1}x_1-b_{i2}x_2-\hat b_{i1}\varphi_1(-r))-\hat b_{i2}\varphi_2(-r)\right)}{1+ x_1+x_2}\\
&<\abs{\vec x}\sum_{i}a_i+ b_2^*\abs{\Vphi(-r)},\;\forall \Vphi\in\C_+.
\end{aligned}
\end{equation}
Let $M>\max\{M_1,M_2\}$, $\vec c=(1,1)$,
$$0<\gamma_b<\min\left\{\frac{b_1^*}{2\sum_{i}a_i},\frac{\sigma_*}2,\frac{b_1^*-b_2^*}{\sum_{i,j}\left(b_{ij}+|\hat b_{ij}|\right)}\right\}, \;\;0<\gamma_0<\frac{b_1^*}2-\gamma_b\sum_{i}a_i,$$
$A_1, A_2$ be such that
$$0<b_2^*+\gamma_b\sum_{i,j}|\hat b_{ij}|<A_2<A_1<b_1^*-\gamma_b\sum_{i,j}b_{ij}\text{ and } A_1-A_2<\frac{b_1^*}2,$$
and $h(\vec x):=1+\abs{\vec x}$,
$\mu$ is be the Dirac delta measure (concentrated) at $\{-r\}$, and
$$
\begin{aligned}
A_0:=&\gamma_0+A_1(1+M)+\gamma_b\left(\sum_{i}a_i+M\sum_{ij}b_{ij}+2\right)+M\sum_{i}a_i\\
&+\sup_{\abs{\vec x}<M}\left\{\dfrac{\sum_{i,j=1}^2\sigma_{ij}x_ix_j}{(1+x_1+x_2)^2}+\dfrac{\sum_{i=1}^2x_i\left(a_i-b_{i1}x_1-b_{i2}x_2\right)}{1+ x_1+x_2}\right\}.
\end{aligned}
$$
Combined with \eqref{5.1-eq1}, \eqref{5.1-eq2}, and \eqref{5.1-eq2-11},
direct calculations  lead to that \eqref{asp1-3} is satisfied and that
Assumption \ref{asp1} holds.
%Similarly, we can check that Assumption \ref{a.extn2} is also satisfied.
Moreover, it is easy to confirm that Assumption \ref{asp-2} and Assumption \ref{asp-unique-ipm}
%,and Assumption \ref{asp-ginverse}
also hold.

Applying
our Theorems in Section \ref{sec:res},
we
have that $\lambda_i(\bdelta^*)=a_i-\dfrac{\sigma_{ii}}2, i=1,2$.
Let $\C^{\circ}_{1+}:=\{(\varphi_1,0)\in\C_+: \varphi_1(s)>0\;\forall s\in[-r,0]\}$ and $\C^{\circ}_{2+}:=\{(0,\varphi_2)\in\C_+: \varphi_2(s)>0\;\forall s\in[-r,0]\}$.
In view of Theorem \ref{thm3.3}, if $\lambda_i(\bdelta^*)>0$, there is a unique invariant probability measure $\pi_i$ on $\C^{\circ}_{i+}$, $i=1,2$.
By Lemma \ref{lm4.1},
we have
$$\lambda_i(\pi_i)=a_i-\dfrac{\sigma_{ii}}2-\int_{\C^{\circ}_{i+}}\left(b_{ii}\varphi_i(0)+\hat b_{ii}\varphi_i(-r)\right)\pi_i(d\Vphi)=0, \text{ where }\Vphi=(\varphi_1,\varphi_2),$$
which implies
\begin{equation}\label{e2-ex1}
\int_{\C^{\circ}_{i+}}\left(b_{ii}\varphi_i(0)+\hat b_{ii}\varphi_i(-r)\right)\pi_i(d\Vphi)=a_i-\dfrac{\sigma_{ii}}{2}.
\end{equation}
Since $\pi_i$ is an invariant probability measure of $\{\vec X_t\}$, it is easy to see that
\begin{equation}\label{ex1-eq55}
\int_{\C^{\circ}_{i+}}\varphi_i(0)\pi_i(d\Vphi)=\lim_{T\to\infty}\dfrac 1T\int_{0}^T X_{i,t}(0)dt=\lim_{T\to\infty}\dfrac 1T\int_{0}^T X_{i}(t)dt,
\end{equation}
where $(X_{1,t},X_{2,t})=\vec X_t$. Similarly,
\begin{equation}\label{ex1-eq56}
\int_{\C^{\circ}_{i+}}\varphi_i(-r)\pi_i(d\Vphi)=\lim_{T\to\infty}\dfrac 1T\int_{0}^T X_{i}(t-r)dt.
\end{equation}
By virtue of \eqref{ex1-eq55} and \eqref{ex1-eq56}, we can prove that
\begin{equation}\label{ex-eq00}
\int_{\C^{\circ}_{i+}}\varphi_i(0)\pi_i(d\Vphi)=\int_{\C^{\circ}_{i+}}\varphi_i(-r)\pi_i(d\Vphi).
\end{equation}
Combining \eqref{e2-ex1} and \eqref{ex-eq00} yields that
$$
\int_{\C^{\circ}_{i+}}\varphi_i(0)\pi_i(d\Vphi)=\int_{\C^{\circ}_{i+}}\varphi_i(-r)\pi_i(d\Vphi)=\dfrac{a_i-\frac{\sigma_{ii}}2}{b_{ii}+\hat b_{ii}}.
$$
Therefore, we have
$$
\begin{aligned}
\lambda_2(\pi_1)&=\int_{\C^{\circ}_{1+}}\left[a_2-\frac{\sigma_{22}}{2}-b_{21}\varphi_1(0)-\hat b_{21}\varphi_1(-r)\right]\pi_1(d\Vphi)\\
&=a_2-\frac{\sigma_{22}}{2}-\left(a_1-\frac{\sigma_{11}}2\right)\cdot\dfrac{b_{21}+\hat b_{21}}{b_{11}+\hat b_{11}},
\end{aligned}
$$
and
$$
\begin{aligned}
\lambda_1(\pi_2)&=\int_{\C^{\circ}_{2+}}\left[a_1-\frac{\sigma_{11}}{2}-b_{12}\varphi_2(0)-\hat b_{12}\varphi_2(-r)\right]\pi_2(d\Vphi)\\
&=a_1-\frac{\sigma_{11}}{2}-\left(a_2-\frac{\sigma_{22}}2\right)\cdot\dfrac{b_{12}+\hat b_{12}}{b_{22}+\hat b_{22}}.
\end{aligned}
$$

If $\lambda_1(\bdelta^*)>0, \lambda_2(\bdelta^*)>0$ and $\lambda_1(\pi_2)>0, \lambda_2(\pi_1)>0$,
any invariant probability measure in $\partial\C_+$ has the form
$\pi=q_0\bdelta^*+q_1\pi_1+q_2\pi_2$ with $0\leq q_0,q_1,q_2$ and $q_0+q_1+q_2=1$.
Then, one has
$\max_{i=1,2}\left\{\lambda_i(\pi)\right\}>0$
for any $\pi$ having the form as above.
As a consequence of Theorem \ref{thm3.3}, there is a unique invariant probability measure $\pi^*$ on $\C^{\circ}_+$.
This result generalizes the results of long-term properties in \cite{Liu17}.

In the above, we considered a $2$-dimension case to illustrate
the idea as well as to simplify the explicit computation.
For the stochastic delay Lotka-Volterra competitive model with $n$-species, our results can still be applied to characterize the long-term behavior of the solution.

\subsection{Stochastic delay Lotka-Volterra predator-prey models}
To continue our study of Lotka-Volterra competitive models, this section is devoted to applying our results to stochastic Lotka-Volterra predator-prey models with time delay. Such models are
frequently used to describe the dynamics of biological systems in which two species interact, one as a predator and the other one as prey.
In this section,
we consider Lotka-Volterra predator-prey system with one prey and two competing predators as follows
\begin{equation}
\begin{cases}
dX_1(t)=X_1(t)\Big\{a_1-b_{11} X_1(t)-b_{12}X_2(t)-b_{13}X_3(t)\\
\hspace{1.5cm}-\hat b_{11}X_1(t-r)-\hat b_{12}X_2(t-r)- \hat b_{13} X_3(t-r)
\Big\}dt+X_1(t)dE_1(t),\\[1ex]
dX_2(t)=X_2(t)\Big\{-a_2+b_{21} X_1(t)-b_{22}X_2(t)-b_{23}X_3(t)\\
\hspace{1.5cm}-\hat b_{21}X_1(t-r)-\hat b_{22}X_2(t-r)- \hat b_{23} X_3(t-r)
\Big\}dt+X_2(t)dE_2(t),\\[1ex]
dX_3(t)=X_3(t)\Big\{-a_3+b_{31} X_1(t)-b_{32}X_2(t)-b_{33}X_3(t)\\
\hspace{1.5cm}-\hat b_{31}X_1(t-r)-\hat b_{32}X_2(t-r)- \hat b_{33} X_3(t-r)
\Big\}dt+X_3(t)dE_3(t),
\end{cases}
\end{equation}
where
$X_1(t)$, $X_2(t)$, and $X_3(t)$
are the densities at time $t$ of the prey, and two predators,
respectively;
$a_1>0$ is the growth rate;
$a_2,a_3>0$ are the death rate of $X_2,X_3$;
$b_{ii}>0, i=1,2,3$ denote the intra-specific competition coefficient of $X_i$;
$b_{ij}\geq 0, i\neq j=1,2,3$, in which $b_{12},b_{13}$ represent the capture rates, $b_{21},b_{31}$ represent the growth from food, and $b_{23}$ and  $b_{32}$ signify the competitions between predators (species 2 and 3);
for each $i,j\in\{1,2,3\}$, $\hat b_{ij}$ is either positive or in $(-b_{ii},0]$;
$r$ is the time delay;
$(E_1(t),E_2(t),E_3(t))^\top=\Gamma^\top\vec B(t)$ with
$\vec B(t)=(B_1(t),B_2(t), B_3(t))^\top$ being a vector of independent standard Brownian motions and
$\Gamma$ being a $3\times 3$ matrix such that
$\Gamma^\top\Gamma=(\sigma_{ij})_{3\times 3}$ is a positive definite matrix.

The model in
 the current setup, 
was considered in \cite{Liu17-dcds}.
However, by
switching the sign of $a_i$ or $b_{ij},i\neq j$, we can obtain a stochastic time-delay
Lotka-Volterra system with the prey and the mesopredator or intermediate predator.
Note that
the case involving a superpredator or top predator, was studied in \cite{Liu17-1,Wu19}, and the stochastic time-delay
Lotka-Volterra system with one predator and two preys was investigated in \cite{Gen17}.

By a
similar calculation as in Subsection \ref{subsec:1}, we can check that \eqref{asp1-3} is satisfied if we let $\vec c=\left(1,\dfrac{b_{12}}{b_{21}},\dfrac{b_{13}}{b_{31}}\right)$ and other parameters be similarly determined as in Section \ref{subsec:1}. Moreover, other assumptions also hold.

We consider the equation on the boundaries $\C_{12+}:=\{(\varphi_1, \varphi_2, 0)\in\C_+: \varphi_1(s), \varphi_2(s)\geq0\;\forall s\in[-r,0]\}$
and
$\C_{13+}:=\{(\varphi_1, 0, \varphi_3)\in\C_+: \varphi_1(s),\varphi_3(s)\geq0,\;\forall s\in[-r,0]\}$.
If $\lambda_1(\bdelta^*)>0$, there is an invariant probability measure $\pi_1$ on $\C^{\circ}_{1+}:=\{(\varphi_1, 0, 0)\in\C_+: \varphi_1(s)>0\;\forall s\in[-r,0]\}$.

In view of Lemma \ref{lm4.1}, we obtain
\begin{equation}\label{e2-ex2}
\int_{\C^{\circ}_{1+}}\left(b_{11}\varphi_1(0)+\hat b_{11}\varphi_1(-r)\right)\pi_1(d\Vphi)=a_1-\dfrac{\sigma_{11}}2.
\end{equation}
Similar to the process of getting \eqref{ex-eq00}, we  obtain from \eqref{e2-ex2} that
$$
\int_{\C^{\circ}_{1+}}\varphi_1(0)\pi_1(d\Vphi)=\int_{\C^{\circ}_{1+}}\varphi_1(-r)\pi_1(d\Vphi)=\dfrac{a_1-\frac{\sigma_{11}}2}{b_{11}+\hat b_{11}}.
$$
Therefore,
$$
\begin{aligned}
\lambda_i(\pi_1)&=\int_{\C^{\circ}_{1+}}\left[-a_i-\frac{\sigma_{ii}}{2}+b_{i1}\varphi_1(0)-\hat b_{i1}\varphi_1(-r)\right]\pi_1(d\Vphi)\\
&=-a_i-\frac{\sigma_{ii}}{2}+\left(a_1-\dfrac{\sigma_{11}}2\right)\cdot\dfrac{b_{i1}-\hat b_{i1}}{b_{11}+\hat b_{11}},\; i=2,3.
\end{aligned}
$$

In case of  $\lambda_1(\bdelta^*)>0$ and $\lambda_2(\pi_1)>0$, Theorem \ref{thm3.3} implies that
there is an invariant probability measure $\pi_{12}$ on $\C^\circ_{12+}.$
In view of Lemma \ref{lm4.1} and \eqref{ex-eq00},
we obtain
$$
\int_{\C^{\circ}_{12+}}\varphi_1(0)\pi_{12}(d\Vphi)=\int_{\C^{\circ}_{12+}}\varphi_1(-r)\pi_{12}(d\Vphi)=A_1,$$
$$ \int_{\C^{\circ}_{12+}}\varphi_2(0)\pi_{12}(d\Vphi)
=\int_{\C^{\circ}_{12+}}\varphi_2(-r)\pi_{12}(d\Vphi)=A_2,$$
where the pair $(A_1, A_2)$ is the unique solution to
$$
\begin{cases}
a_1-\frac{\sigma_{11}}{2}-\left(b_{11}+\hat b_{11}\right)A_1-\left(b_{12}+\hat b_{12}\right)A_2=0,\\[1ex]
-a_2-\frac{\sigma_{22}}{2}+\left(b_{21}-\hat b_{21}\right)A_1-\left(b_{22}+\hat b_{22}\right)A_2=0.
\end{cases}
$$
In this case,
\begin{align*}
\lambda_3(\pi_{12})&=\int_{\C^{\circ}_{12+}}\Big[-a_3-\frac{\sigma_{33}}{2}+\left(b_{31}\varphi_1(0)-\hat b_{31}\varphi_1(-r)\right)\\
&\hspace{3cm}-\left(b_{32}\varphi_2(0)+\hat b_{32}\varphi_2(-r)\right)\Big]\pi_{12}(d\Vphi)\\
&=-a_3-\frac{\sigma_{33}}{2}+\left(b_{31}-\hat b_{31}\right)A_1-\left(b_{32}+\hat b_{32}\right)A_2.
\end{align*}
Similarly, if $\lambda_1(\bdelta^*)>0$ and $\lambda_3(\pi_1)>0$, by Theorem \ref{thm3.3},
there is an invariant probability measure $\pi_{13}$ on $\C^\circ_{13+}$
and
\begin{align*}
\lambda_2(\pi_{13})&=\int_{\C^{\circ}_{13+}}\Big[-a_2-\frac{\sigma_{22}}{2}+\left(b_{21}\varphi_1(0)-\hat b_{21}\varphi_1(-r)\right)\\
&\hspace{3cm}-\left(b_{23}\varphi_3(0)+\hat b_{23}\varphi_3(-r)\right)\Big]\pi_{13}(d\Vphi)\\
&=-a_2-\frac{\sigma_{22}}{2}+\left(b_{21}-\hat b_{21}\right)\hat A_1-\left(b_{32}+\hat b_{23}\right)\hat A_3,
\end{align*}
where
$(\hat A_1, \hat A_3)$ is the unique solution to
$$
\begin{cases}
a_1-\frac{\sigma_{11}}{2}-\left(b_{11}+\hat b_{11}\right)\hat A_1-\left(b_{13}+\hat b_{13}\right)\hat A_3=0,\\[1ex]
-a_3-\frac{\sigma_{33}}{2}+\left(b_{31}-\hat b_{31}\right)\hat A_1-\left(b_{33}+\hat b_{33}\right)\hat A_3=0.
\end{cases}
$$

Because of the ergodic decomposition theorem,
every invariant probability measure on $\partial \C_+$
is a convex combination of $\delta^*, \pi_1,\pi_{12},\pi_{13}$ (when these measures exist). As a consequence,
some computations for the Lyapunov exponents
with respect to a convex combination of these ergodic measures together with an application of Theorem \ref{thm3.3} yield that
there exists a unique invariant probability measure in $\C_+^\circ$
if one of the following conditions is satisfied:
\begin{itemize}
	\item $\lambda_1(\bdelta^*)>0$, $\lambda_2(\pi_1)>0, \lambda_3(\pi_1)<0$ and $\lambda_3(\pi_{12})>0$.
	\item $\lambda_1(\bdelta^*)>0$, $\lambda_2(\pi_1)<0, \lambda_3(\pi_1)>0$    and $\lambda_2(\pi_{13})>0$.
	\item $\lambda_1(\bdelta^*)>0$, $\lambda_2(\pi_1)>0, \lambda_3(\pi_1)>0$, $\lambda_3(\pi_{12})>0$,   and $\lambda_2(\pi_{13})>0$.
\end{itemize}
The above assertions
generalize the results in \cite{Liu17-dcds}. Moreover, if we switch the sign of $a_i$ or $b_{ij},i\neq j$, we obtain another modifications of Lotka-Volterra prey-predator equation as we mentioned at the beginning of this section with modification of %slightly
the
above characterization, which improve the results in \cite{Gen17,Liu17-1,Wu19}.

Confining
our analysis to $\C_{12+}$ (this describes the evolution of one predator and its prey), we get
\begin{equation}\label{e3-ex2}
\begin{cases}
dX_1(t)=X_1(t)\left\{a_1-b_{11} X_1(t)-\hat b_{11}X_1(t-r)-b_{12}X_2(t)-\hat b_{12}X_2(t-r)
\right\}dt\\
\hspace{9cm}+X_1(t)dE_1(t),\\[1ex]
dX_2(t)=X_2(t)\left\{-a_2+b_{21}X_1(t)+\hat b_{21}X_1(t-r)-b_{22}X_2(t)-\hat b_{22}X_2(t-r)
\right\}dt\\
\hspace{9cm}+X_2(t)dE_2(t).
\end{cases}
\end{equation}
This further leads to that if $\lambda_1(\bdelta^*)>0,\lambda_2(\pi_1)>0$,
there exists a unique invariant probability measure of
\eqref{e3-ex2}
on $\C^\circ_{12+}$,
which improves the results in \cite{Liu13}.

\subsection{Stochastic delay
	replicator equations}
In evolutionary game theory, originally,
a replicator equation is a deterministic monotone, nonlinear, and non-innovative game dynamic system. Such a deterministic system has been expanded to systems with stochastic perturbations. In this section,
we consider the replicator dynamics for a game with $n$ strategies, 
involving
social-type time delay (see, e.g., \cite{AM04} for details of such delays)
and white noise perturbation.
The system of interest
can be expressed as %following equation
\begin{equation}\label{ex5-eq1}
\begin{cases}
dx_i(t)=x_i(t)\left(f_i(\vec x(t-r))-\dfrac 1X\displaystyle\sum_{j=1}^n x_j(t)f_j(\vec x(t-r))\right)dt\\
\hspace{2.5cm}+x_i(t)\left(\sigma_idB_i(t)-\dfrac 1X\displaystyle\sum_{j=1}^n \sigma_j x_jdB_j(t)\right);\;i=1,\dots,n,\\
\vec x(s)=\vec x_0(s);\;t\in[-r,0],
\end{cases}
\end{equation}
where
$X$ is the size of the populations;
$x_i(t)$
is the portion of population that has selected the $i^{th}$ strategy and the distribution of the whole population among the strategy;
the fitness functions $f_i(\cdot):\R^n_+\to\R$, $i=1,\dots,n$ are the payoffs obtained by the individuals playing the $i^{th}$ strategy;
$r$ is the time delay;
and  $\vec x_0(s)\in \Delta_X:=\{\vec x\in\R_+^n:\sum_{i=1}^nx_i=X\}$ for all $s\in[-r,0]$ is the initial value.

The replicator
equation was
introduced in 1978 by Taylor and Jonker in \cite{Tay78}.
Since then significant
contributions have been made in biology \cite{Hof98,Now04},  economics \cite{Wei97},
%as well as in
and
optimization and control for a variety of systems \cite{Bom00, Oba14,Ram10, Tem10}.
Much attention has been devoted to studying their properties.
For instance,
when $f_i(\cdot):\R^n_+\to \R, i=1,\dots,n$ are linear mappings, equation \eqref{ex5-eq1} without time delay was studied in \cite{Imh09,Imh05}. Moreover, the deterministic version of equation \eqref{ex5-eq1} was studied in \cite{AM04,Oba16}.

By a similar argument as in \cite{Oba16,Wei97}, we can show that $\Delta_X$ remains invariant\a.s As a consequence, our assumptions are verified.
Hence, our results in Theorem \ref{thm3.3}
%, Theorem \ref{thm4.1}, and Theorem \ref{thm4.2}
 holds for \eqref{ex5-eq1}.
To demonstrate, for better visualization, we apply
%Now, we make some applications of
our results to some
%be more discrete
low-dimensional systems.
First, we consider equation \eqref{ex5-eq1} in case of two dimensions.
Define
$$
\C_+^X:=\{(\varphi_1,\varphi_2):\varphi_1(s)+\varphi_2(s)=X\text{ and }\varphi_1(s),\varphi_2(s)\geq 0\text{ for all }s\in[-r,0]\},
$$
$$\partial\C_+^X:=\{(\varphi_1,\varphi_2)\in\C_+^X: \norm{\varphi_1}=0\text{ or }\norm{\varphi_2}=0\},$$
$$\C_{+}^{X,\circ}:=\{(\varphi_1,\varphi_2)\in\C_+^X: \varphi_1(s), \varphi_2(s)>0\text{ for all }s\in[-r,0]\}.$$
In this case, it is clear that there are two invariant probability measures on the boundary $\partial \C_+^X$, which are $\bdelta_1$ and $\bdelta_2$ concentrating on $(X,0)$ and $(0,X)$, respectively, where $0,X$ are understood to be constant functions.
We have
$$\lambda_1(\bdelta_2)=f_1((0,X))-f_2((0,X))-\dfrac{\sigma_1^2+\sigma_2^2}2,$$
$$\lambda_2(\bdelta_1)=f_2((X,0))-f_1((X,0))-\dfrac{\sigma_1^2+\sigma_2^2}2.$$
By Theorem \ref{thm3.3}, in case of \eqref{ex5-eq1} of $2$-dimensional systems,
if $\lambda_1(\bdelta_2)>0$ and $\lambda_2(\bdelta_1)>0$, there is a unique invariant probability measure of \eqref{ex5-eq1} on $\C^{X,\circ}_{+}$.

Next, we consider \eqref{ex5-eq1} in three dimensions. Similarly, we also define the following set
$$
\C_+^X:=\{(\varphi_1,\varphi_2,\varphi_3):\varphi_1(s)+\varphi_2(s)+\varphi_3(s)=X\text{ and }\varphi_1(s),\varphi_2(s),\varphi_3(s)\geq 0\text{ for all }s\in[-r,0]\},
$$
$$\partial\C_+^X:=\C_{12+}^X\cup\C_{23+}^X\cup\C_{13+}^X,$$
$$
\C_{ij+}^X:=\{(\varphi_1,\varphi_2,\varphi_3)\in\C_+^X: \norm{\varphi_k}=0,k\neq i,j\}, \text{ for }i\neq j\in\{1,2,3\},$$
$$\C_{+}^{X,\circ}:=\{(\varphi_1,\varphi_2,\varphi_3)\in\C_+^X: \varphi_1(s), \varphi_2(s), \varphi_3(s)>0\text{ for all }s\in[-r,0]\}.$$
Denote by $\bdelta_1,\bdelta_2,\bdelta_3$ the invariant probability measures on the boundary $\partial\C_+^X$ of \eqref{ex5-eq1}, concentrating on $(X,0,0)$, $(0,X,0)$, and $(0,0,X)$, respectively. We have
$$\lambda_i(\bdelta_1)=f_i((X,0,0))-f_1((X,0,0))
-\dfrac{\sigma_1^2+\sigma_i^2}2,i=2,3,$$
$$\lambda_j(\bdelta_2)=f_j((0,X,0))-f_2((0,X,0))
-\dfrac{\sigma_2^2+\sigma_j^2}2,j=1,3,$$
and
$$\lambda_k(\bdelta_3)=f_k((0,0,X))-f_3((0,0,X))
-\dfrac{\sigma_3^2+\sigma_k^2}2,k=1,2.$$
If $\max_{j=1,3}\lambda_j(\bdelta_2)>0$ and $\max_{k=1,2}\lambda_k(\bdelta_3)>0$,
there is a unique invariant probability measure on $\C_{23+}^X$, denoted by $\pi_{23}$.
When $\pi_{23}$ exists, we have
$$
\begin{aligned}
\lambda_1(\pi_{23})=-\dfrac{\sigma_1^2}2
+\int_{\C_{23+}^X}\bigg(f_1(\Vphi)&-\dfrac{2X\varphi_2(0)
	f_2(\Vphi)+\sigma_2^2\varphi_2^2(0)}{X^2}\\
&
-\dfrac{2X\varphi_3(0)f_3(\Vphi)
	+\sigma_3^2\varphi_3^2(0)}{X^2}\bigg)\pi_{23}(d\Vphi).
\end{aligned}
$$
%\end{align*}
By Lemma \ref{lm4.1} and $\lambda_2(\pi_{23})=\lambda_3(\pi_{23})=0$,
%and then
we have
\begin{align*}
\int_{\C_{23+}^X}&\left(\dfrac{2X\varphi_2(0)f_2(\Vphi)+\sigma_2^2\varphi_2^2(0)}{2X^2}+\dfrac{2X\varphi_3(0)f_3(\Vphi)+\sigma_3^2\varphi_3^2(0)}{2X^2}\right)\pi_{23}(d\Vphi)\\
&=\dfrac{\sigma_2^2}2+\int_{\C_{23+}^X}f_2(\Vphi)\pi_{23}(d\Vphi)\\
&=\dfrac{\sigma_3^2}2+\int_{\C_{23+}^X}f_3(\Vphi)\pi_{23}(d\Vphi).
\end{align*}
As a result,
\begin{align*}
\lambda_1(\pi_{23})&=-\dfrac{\sigma_1^2+\sigma_2^2}2
+\int_{\C_{23+}^X}\left(f_1(\Vphi)-f_2(\Vphi)\right)\pi_{23}(d\Vphi)\\
&=-\dfrac{\sigma_1^2+\sigma_3^2}2+\int_{\C_{23+}^X}\left(f_1(\Vphi)-f_3(\Vphi)\right)\pi_{23}(d\Vphi).
\end{align*}
The conditions to guarantee the existence of the unique invariant  probability measure $\pi_{12},\pi_{13}$ on the boundary $\C_{12+}^X,\C_{13+}^X$ are similarly obtained and
$\lambda_2(\pi_{13}),\lambda_3(\pi_{12})$ can be
computed similar to $\lambda_1(\pi_{23})$.
Therefore, we have the following classification for the
long-run
%of the
solution of \eqref{ex5-eq1} in three dimensions as following. System \eqref{ex5-eq1} admits a unique invariant probability measure on $\C_+^{X,\circ}$ if
\begin{itemize}
	\item $\max_{i=2,3}\lambda_i(\bdelta_1)>0$, $\max_{j=1,3}\lambda_j(\bdelta_2)>0$, $\max_{k=1,2}\lambda_k(\bdelta_3)>0$ and $\lambda_1(\pi_{23})>0$, $\lambda_2(\pi_{13})>0$, $\lambda_3(\pi_{12})>0$.
\end{itemize}
The (explicit) condition for persistence of \eqref{ex5-eq1} in $n$-dimensions is more complex. However, our results (Theorem \ref{thm3.3}) still holds and will be computable in practice under suitable conditions.
Moreover, if $r=0$ (i.e., there is no time delay) and $f_i(\cdot)$, $i=1,\dots,n$ are linear, the condition of the persistence of \eqref{ex5-eq1} in this section is equivalent to results in \cite{Imh09,Imh05}.

\subsection{Stochastic delay epidemic SIR models}\label{ex-SIR}
The SIR model is one of the basic building blocks of compartmental models, from which many
infectious disease
models are derived. The model consists of three compartments, S for the number of susceptible, I for the number of infectious, and R for the number of recovered (or immune).
First introduced by Kermack and
McKendrick in \cite{Ker27,Ker32}, the models are deemed effective to depict the spread of many common diseases with permanent immunity such as rubella, whooping cough, measles, and smallpox. A variety of modifications of original equation are introduced due to the complexity of environment. Much attention
has been devoted to analyzing the behavior of these systems; for example, see \cite{Die16, DN18} and the references therein. In this section, we study the stochastic epidemic SIR models with  time delay.

To start, we consider the equation with linear incidence rate of the following
form
\begin{equation}\label{ex3-eq1}
\begin{cases}
dS(t)=\left(a-b_1 S(t)-c_1 I(t)S(t)-c_2I(t)S(t-r)\right)dt+S(t)dE_1(t),\\[1ex]
dI(t)=\left(-b_2I(t)+c_1 I(t)S(t)+c_2I(t)S(t-r)\right)dt+I(t)dE_2(t),
\end{cases}
\end{equation}
where
$S(t)$ is the density of susceptible individuals,
$I(t)$ is the density of infected individuals,
$a>0$ is the recruitment rate of the population,
$b_i>0$, $i=1,2$ are the death rates,
$c_i>0$, $i=1,2$ are the incidence rates,
$r$ is the delayed time,
$(E_1(t),E_2(t))^\top=\Gamma^\top\vec B(t)$ with
$\vec B(t)=(B_1(t), B_2(t))^\top$ being a vector of independent standard Brownian motions, and
$\Gamma$ being a $2\times 2$ matrix such that
$\Gamma^\top\Gamma=(\sigma_{ij})_{2\times 2}$ is a positive definite matrix.
It is well-known that
%it has been noted that
the dynamics of
recovered individuals have no effect on the disease transmission dynamics and that is why we only consider the dynamics of $S(t),I(t)$ in \eqref{ex3-eq1}.

Although equation \eqref{ex3-eq1} does not have the exact form as in \eqref{Kol-Eq}, we can use the same idea and the same method to obtain similar results. First, we consider the equation on the boundary $\{(\varphi_1,0):\varphi_1(s)\geq 0\;\forall s\in[-r,0]\}$ and let $\hat S(t)$ be the solution of the equation on this boundary as following
\begin{equation}\label{ex3-eq2}
d\hat S(t)=\left(a-b_1 \hat S(t)\right)dt+\hat S(t)dE_1(t).
\end{equation}
Since the drift coefficient of this equation is negative if $\hat S(t)$ is sufficiently large and positive, if $\hat S(t)$ is sufficiently small, %by standard argument ({\color{red}see e.g., \cite{}})
we can show that there is a unique invariant probability measure $\pi$ of \eqref{ex3-eq1} on $\C_{1+}^\circ:=\{(\varphi_1,0):\varphi_1(s)> 0\;\forall s\in[-r,0]\}$. On the other hand, since $\lambda_2(\bdelta^*)=-b_2-\dfrac{\sigma_{22}}2<0$, there is no invariant probability measure in $\C_{2+}^\circ:=\{(0,\varphi_2):\varphi_2(s)>0;\forall s\in [-r,0]\}$.

Hence, we define the following threshold
\begin{equation}\label{ex3-eq3}
\lambda(\pi)=-b_2-\frac {\sigma_{22}}2+\int_{\C_{1+}^\circ}\left(c_1\varphi_1(0)+c_2\varphi_1(-r)\right)\pi(d\Vphi),
\end{equation}
whose sign will be able to characterize the permanence and extinction.
As an application of Lemma \ref{lm4.1}, we get
\begin{equation}\label{ex3-eq4}
\int_{\C_{1+}^\circ}\varphi_1(0)\pi(d\Vphi)=\frac a{b_1}.
\end{equation}
By \eqref{ex-eq00}, we have that
$$
\int_{\C_{1+}^\circ}\varphi_1(-r)\pi(d\Vphi)=\int_{\C_{1+}^\circ}\varphi_1(0)\pi(d\Vphi)=\dfrac{a}{b_1}.
$$
Therefore, under this condition, we obtain from \eqref{ex3-eq3} and \eqref{ex3-eq4} that
\begin{equation*}%\label{ex3-eq5}
\lambda(\pi)=-b_2-\frac {\sigma_{22}}2+\dfrac{a(c_1+c_2)}{b_1}.
\end{equation*}
Using the same idea and techniques, it is possible to obtain similar results to Theorem \ref{thm3.3}
%, Theorem \ref{thm4.1}, and Theorem \ref{thm4.2}
for equation \eqref{ex3-eq1}.
We have that
if $\lambda(\pi)>0$, \eqref{ex3-eq1} has a unique invariant probability measure in $\C_+^\circ$.
This characterization is equivalent to the result in \cite{QLiu16-1,QLiu16}.

In the above, we consider the linear incidence to make our computations be more explicit. The characterizations still hold for the following stochastic delay SIR epidemic model with general incidence rate
\begin{equation}\label{ex3-eq5}
\begin{cases}
dS(t)=\left(a-b_1 S(t)-I(t)f_1\big(S(t),S(t-r),I(t),I(t-r)\big)\right)dt+S(t)dE_1(t),\\[1ex]
dI(t)=\left(-b_2I(t)+I(t)f_2\big(S(t),S(t-r),I(t),I(t-r)\big)\right)dt+I(t)dE_2(t),
\end{cases}
\end{equation}
where
%$S_t,I_t$ are the segment functions, i.e $S_t\in C([-r,0],\R): S_t(s)=S(t+s),\;s\in[-r,0]$ and $I_t$ is similarly defined;
$f_i:\R^4\to \R, i=1,2$ are the incidence functions satisfying
\begin{itemize}
	\item $f_1(0,0,i_1,i_2)=f_2(0,0,i_1,i_2)=0$.
	\item there exists some $\kappa\in (0,\infty)$ such that for all $\Vphi\in\C_+$
	$$
	\begin{aligned}
	f_2\big(\varphi_1(0),&\varphi_1(-r),\varphi_2(0),\varphi_2(-r)\big)\leq \kappa f_1\big(\varphi_1(0),\varphi_1(-r),\varphi_2(0),\varphi_2(-r)\big)\\
	&\leq\kappa^2\left(1+\abs{\Vphi(0)}+\abs{\Vphi(-r)}\right).
	\end{aligned}
	$$
	\item $f_2(s_1,s_2,i_1,i_2)$ is non-decreasing in $s_1,s_2$ and is non-increasing in $i_1,i_2$.
\end{itemize}
Almost all incidence functions used in the literature, e.g., linear functional response, Holling type II functional response, Beddington-DeAngelis functional response, etc., satisfy the above conditions.
In the general case,
the system has a unique invariant probability measure in $\C_+^\circ$ if $\lambda(\pi)>0$, where
$\lambda(\pi)$ is defined as follows
$$
\lambda(\pi)=-b_2-\frac {\sigma_{22}}2+\int_{\C_{1+}^\circ}f_2\big(\varphi_1(0),\varphi_1(-r),\varphi_2(0),\varphi_2(-r)\big)\pi(d\Vphi),
$$
where
$\Vphi=(\varphi_1,\varphi_2)$ and $\pi$ is the invariant probability measure of \eqref{ex3-eq2}. These results significantly generalize and improve that of
%the results in
\cite{Che09,DN17,QLiu16-2,Mah17}.

\subsection{Stochastic delay chemostat models}\label{ex-che}
A chemostat is a bio-reactor. In a chemostat, fresh medium is continuously added, and culture liquid containing left-over nutrients, metabolic end products, and microorganisms are continuously removed at the same rate to keep a constant culture volume.
The chemostat model is based on a
technique introduced by Novick and
Szilard in \cite{Nov50} and plays an important role in microbiology, biotechnology, and
population biology.
This
section is devoted to studying a model of  $n$-microbial populations competing for a single nutrient in a chemostat with delay in uptake conversion and under effects of white noises. Precisely, the model is described by the following system of stochastic functional differential equations
\begin{equation}\label{ex4-eq1}
\begin{cases}
dS(t)=\left(1-S(t)+aS(t-r)-\displaystyle\sum_{i=1}^nx_i(t)p_i(S(t))\right)dt+S(t)dE_0(t),\\[2ex]
dx_i(t)=x_i(t)\left(p_i(S(t-r))-1\right)dt+x_i(t)dE_i(t),\;i=1,\dots,n,\\
%dx_2(t)=x_2(t)\left(p_2(S(t-r))-1\right)dt+x_2(t)dE_2(t),
\end{cases}
\end{equation}
where
$S(t)$ is the concentration of nutrient at time $t$;
$0\leq a< 1$ is a constant;
$x_i(t),i=1,\dots,n$ are the concentrations of the competing microbial populations;
$p_i(S),i=1,\dots,n$ are the density-dependent uptakes of nutrient by population $x_i$;
$r$ is the delayed time; and
$(E_0(t),\dots,E_n(t))^\top=\Gamma^\top\vec B(t)$ with
$\vec B(t)=(B_0(t),\dots, B_n(t))^\top$ being a vector of independent standard Brownian motions and
$\Gamma$ being a $(n+1)\times (n+1)$ matrix such that
$\Gamma^\top\Gamma=(\sigma_{ij})_{(n+1)\times (n+1)}$ is a positive definite matrix. Moreover,
%due to the different notation,
in this section,
%subsection
$\C:=\C([-r,0],\R^{n+1})$ instead of $\C([-r,0],\R^n)$.
The deterministic version of \eqref{ex4-eq1} is studied and the long-time behavior is characterized in \cite{Ell94,Fre89,Wol97}.
Recently, much attention is devoted to studying the related stochastic systems; see \cite{Sun18,Sun18-1,Zha19}.

It is similar to Section \ref{ex-SIR}, if we assume that $p_i:\R\to\R, i=1,\dots,n$ satisfying non-decreasing and bounded properties and $p_i(0)=0$, then our Assumptions hold. Therefore, our results in this paper can be applied to \eqref{ex4-eq1}.

Before obtaining the results in multi-dimensional systems,
%dimension,
we consider $n=1,2$.
If $n=1$, there is only one population $x_1$ together with the nutrient $S(t)$.
Similar to Section \ref{ex-SIR}, there is no invariant probability measure of $(S_t,x_{1t})$ in
$\C_{1+}^\circ:=\{(0,\varphi_1)\in\C_+:\varphi_1(s)>0,\forall s\in [-r,0]\},$
where $x_{1t}$ is the memory segment function of $x_1(t)$.
Moreover, there is a unique invariant probability measure $\pi_0$ in $\C_{0+}^\circ:=\{(\varphi_0,0)\in\C_+:\varphi_0(s)>0,\forall s\in [-r,0]\}$.
Hence, it is easy to see that for any invariant probability measure $\pi$ in $\partial\C_+$, we have
\begin{equation*}
\lambda_1(\pi)=\lambda_1(\pi_0)=-1-\dfrac {\sigma_{11}}2+\int_{\C_{0+}^\circ}p_1(\varphi_0(-r))\pi_0(d\Vphi).
%,\;i=1,\dots,n.
\end{equation*}
By applying our result, if $\lambda_1(\pi_0)>0$ then $(S_t,x_{1t})$ admits a unique invariant probability measure in $\C_+^\circ$.

We next reveal the characterization of the longtime behavior in the case $n=2$, which is similar to the case of $n=1$.
%that
There is no invariant probability measure in
$\C_{i+}^\circ:=\{(0,\varphi_1,\varphi_2)\in\C_+:\norm{\varphi_j}=0, j\neq i \text{ and } \varphi_i(s)>0,\forall s\in [-r,0]\},$
and there is a unique measure $\pi_0$ in
$\C_{0+}^\circ:=\{(\varphi_0,0,0)\in\C_+:\varphi_0(s)>0,\forall s\in [-r,0]\}$.
As characterized in the case $n=1$, if $\lambda_i(\pi_0)> 0$, where
\begin{equation*}
\lambda_i(\pi_0)=-1-\dfrac {\sigma_{ii}}2+\int_{\C_{0+}^\circ}p_i(\varphi_0(-r))\pi_0(d\Vphi),\;i=1,2,%\dots,n,
\end{equation*}
then there is a unique invariant probability measure $\pi_{0i}$ in
$C_{0i+}^\circ:=\{(\varphi_0,\varphi_1,\varphi_2)\in\C_+:\norm{\varphi_j}=0, j\neq i \text{ and } \varphi_0(s), \varphi_i(s)>0,\forall s\in [-r,0]\}$. Hence, let
\begin{equation*}
\lambda_j(\pi_{0i})=-1-\dfrac {\sigma_{jj}}2+\int_{\C_{0+}^\circ}p_j(\varphi_0(-r))\pi_{0i}(d\Vphi),\;j\neq i.
\end{equation*}
The persistence is classified as follows.
 The $(S_t,x_{1t},x_{2t})$ admits a unique invariant probability measure in $\C_+^\circ$
if $\lambda_1(\pi_0)>0$, $\lambda_2(\pi_0)>0$, $\lambda_1(\pi_{02})>0$, and $\lambda_2(\pi_{01})>0$.

The two examples in low dimension ($n=1,2$) provide a scheme to construct recursively the characterization of the longtime behavior of \eqref{ex4-eq1} in higher dimensions.
%via lower dimension.
It is difficult to show concretely in case of general functions $p_i(\cdot)$, but it
is computable in certain examples.
These classifications improve the results in \cite{Sun18,Zha19}.

\begin{rem}{\rm
		In fact, in all the examples in Sections  \ref{subsec:1}-\ref{ex-che},
		similar results can be  obtained for multi-delays or distributed delays. We used
		a single delay in this Section
		for
		simplifying the notation and calculations so as to present
		the main ideas without notation complication.
		On the other hand, if $r=0$, i.e., there is no time delay, the above results are consistent with and/or improve the existing results in the literature.
}\end{rem}

\end{document}